\documentclass[11pt]{amsart}

\usepackage{scrextend}

\newcommand\legendre[2]{{#1\overwithdelims () #2}}
\newenvironment{romanenum}{\hfill \begin{enumerate} }{\end{enumerate}}
\newenvironment{alphenum}{\hfill \begin{enumerate} }{\end{enumerate}}

\makeatletter
\newcommand\footnoteref[1]{\protected@xdef\@thefnmark{\ref{#1}}\@footnotemark}
\makeatother

\usepackage{amsmath,amsfonts,amssymb,amsthm,url,verbatim}
\usepackage[dvips]{graphicx}
\usepackage[croatian,english]{babel}
\usepackage{amsfonts}
\usepackage[utf8]{inputenc}
\usepackage[T1]{fontenc}
\usepackage{amsmath}
\usepackage{amssymb}
\usepackage{latexsym}
\usepackage[usenames,dvipsnames]{xcolor}
\usepackage{soul}

\usepackage{booktabs}
\usepackage{longtable}
\usepackage{multirow}

\usepackage[backref]{hyperref}

\usepackage{ifthen}
\usepackage{graphpap}
\newcommand{\cC}{{\mathcal{C}}}

\newcommand{\Q}{\mathbb{Q}}

\newcommand{\Z}{\mathbb{Z}}

\newcommand{\F}{\mathbb{F}}
\newcommand{\Qbar}{{\overline{\mathbb Q}}}
\DeclareMathOperator{\GL}{GL}

\DeclareMathOperator{\tors}{tors}

\newtheorem{tm}{Theorem}[section]
\newtheorem{proposition}[tm]{Proposition}
\newtheorem{lemma}[tm]{Lemma}
\newtheorem{corollary}[tm]{Corollary}
\newtheorem{conjecture}[tm]{Conjecture}
\theoremstyle{definition}

\theoremstyle{remark}

\newtheorem{remark}[tm]{Remark}

\DeclareMathOperator{\Gal}{Gal}

\newcommand{\Aut}{\operatorname{Aut}}

\def\diam#1{\langle#1\rangle}
\newcommand{\lmfdbec}[3]{\href{http://www.lmfdb.org/EllipticCurve/Q/#1/#2/#3}{#1#2#3}}

\title[Growth of torsion groups of elliptic curves]{Growth of torsion groups of\\ elliptic curves upon base change}
\author{Enrique Gonz\'alez--Jim\'enez}
\address{Universidad Aut{\'o}noma de Madrid, Departamento de Matem{\'a}ticas, Madrid, Spain}
\email{enrique.gonzalez.jimenez@uam.es}
\urladdr{http://matematicas.uam.es/~enrique.gonzalez.jimenez}
\author{Filip Najman}
\address{University of Zagreb, Bijeni\v{c}ka cesta 30, 10000 Zagreb, Croatia}
\email{fnajman@math.hr}
\urladdr{http://web.math.pmf.unizg.hr/~fnajman/}
\thanks{The first author was partially supported by the grant MTM2015--68524--P. The second author gratefully acknowledges support from the QuantiXLie Center of Excellence.}
\date{\today}

\begin{document}

\begin{abstract}
 We study how the torsion of elliptic curves over number fields grows upon base change, and in particular prove various necessary conditions for torsion growth. For a number field $F$, we show that for a large set of number fields $L$, whose Galois group of their normal closure over $F$ has certain properties, it will hold that $E(L)_{\tors}=E(F)_{\tors}$ for all elliptic curves $E$ defined over $F$.

Our methods turn out to be particularly useful in studying the possible torsion groups $E(K)_{\tors}$, where $K$ is a number field and $E$ is a base change of an elliptic curve defined over $\Q$. Suppose that $E$ is a base change of an elliptic curve over $\Q$ for the remainder of the abstract. We prove that $E(K)_{\tors}=E(\Q)_{\tors}$ for all elliptic curves $E$ defined over $\Q$ and all number fields $K$ of degree $d$, where $d$ is not divisible by a prime $\leq 7$. Using this fact, we determine all the possible torsion groups $E(K)_{\tors}$ over number fields $K$ of prime degree $p\geq 7$. We determine all the possible degrees of $[\Q(P):\Q]$, where $P$ is a point of prime order $p$ for all $p$ such that $p\not\equiv 8 \pmod 9$ or $\left( \frac{-D}{p}\right)=1$ for any $D\in \{1,2,7,11,19,43,67,163\}$; this is true for a set of density $\frac{1535}{1536}$ of all primes and in particular for all $p<3167$. Using this result, we determine all the possible prime orders of a point $P\in E(K)_{\tors}$, where $[K:\Q]=d$, for all $d\leq 3342296$. Finally, we determine all the possible groups $E(K)_{\tors}$, where $K$ is a quartic number field and $E$ is an elliptic curve defined over $\Q$ and show that no quartic sporadic point on a modular curves $X_1(m,n)$ comes from an elliptic curve defined over $\Q$.

\end{abstract}

\maketitle
\tableofcontents
\section{Introduction}
Let $E/F$ be an elliptic curve defined over a number field $F$. The Mordell-Weil Theorem states that the set $E(F)$ of $F$-rational points is a finitely generated abelian group. Denote by $E(F)_{\tors}$, the torsion subgroup of $E(F)$. It is well known that the possible torsion groups are of the shape $\cC_m\times\cC_n$ for two positive integers $n,m$, where $m$ divides $n$ and where $\cC_k$ is a cyclic group of order $k$.

The purpose of this paper is to shed light on how the torsion group $E(F)_{\tors}$ grows to $E(L)_{\tors}$, where $E$ is defined over a number field $F$ and $L$ is a field extension of $F$. We will mostly be interested in the case where $L$ is also a number field, although some of our results apply to infinite extensions of $F$ - for example in Corollary \ref{cor:zp} we prove that the torsion of elliptic curves over $\Q$ does not grow in $\Z_p$-extensions of $\Q$ for $p\geq 11$. Our results give necessary conditions on the degrees and Galois groups of extensions of number fields $L/F$ depending on the growth from $E(F)_{\tors}$ to $E(L)_{\tors}$. The most general results of this type are proved Section \ref{sec2}, where $L/F$ is an arbitrary extension of number fields.

Naturally, we can obtain much more precise results if we specialize to the case of $F=\Q$. We will do this in Section \ref{sec:RQD}. Suppose till the end of the introduction that $E$ is always defined over $\Q$. For $P\in E(\overline \Q)$ define $\Q(P)$ to be the number field obtained by adjoining the coordinates of $P$ to $\Q$. We study the problem of determining, for a given prime $p$, all the possible values of the degree of $[\Q(P):\Q]$, for all points $P\in E(\overline \Q)$ of order $p$ on all elliptic curves $E/\Q$. This is done by studying the lengths of orbits of $E[p]$ under the action of the absolute Galois group $\Gal(\Qbar/\Q)$. Here we rely heavily on the knowledge of the possible images $G_E(p)$ of mod $p$ Galois representations attached to elliptic curves over $\Q$. While all the possibilities for $G_E(p)$ are not yet known in general, the state of the knowledge is good enough for our results, stated in Theorem \ref{dQ}, to be unconditional for all $p$ such that $p\not\equiv 8 \!\pmod 9$ or $\left( \frac{-D}{p}\right)=1$ for any $D\in \{1,2,7,11,19,43,67,163\}$. The set of primes satisfying these conditions has natural density $\frac{1535}{1536}$ and the smallest prime that does not satisfy these conditions is $p=3167$. Should Serre's uniformity conjecture \ref{serre_conj} or its weaker version, Conjecture \ref{uniformity_conj}, be proved our results will hold true for all primes $p$.

An important problem in the theory of elliptic curves is to characterize the possible torsion groups over a given number field, or over number fields of given degree. Before we state known results about this problem, we introduce some notation. For a given positive integer $d$,
\begin{itemize}
\item Let $\Phi(d)$ be the set of possible isomorphism classes of groups $E(K)_{\tors}$, where $K$ runs through all number fields $K$ of degree $d$ and $E$ runs through all elliptic curves over $K$.

\item Let $\Phi_{\Q}(d)\subseteq \Phi(d)$ be the set of possible isomorphism classes of groups $E(K)_{\tors}$, where $K$ runs through all number fields $K$ of degree $d$ and $E$ runs through all elliptic curves defined over $\Q$.

\item Let $\Phi^{\infty}(d)$ be the subset of isomorphism classes of groups $\Phi(d)$  that occur infinitely often. More precisely, a torsion group $G$ belongs to  $\Phi^{\infty}(d)$ if there are infinitely many elliptic curves $E$, non-isomorphic over $\overline{\Q}$, such that $E(K)_{\tors}\simeq G$.

\item Let $R_{\Q}(d)$ be the set of all primes $p$ such that there exists a number field $K$ of degree $d$, an elliptic curve $E/ \Q$ such that there exists a point of order $p$ on $E(K)$.
\end{itemize}

Mazur \cite{maz} determined $\Phi(1)$, while Kamienny, Kenku and Momose \cite{km,kam} determined $\Phi(2)$.  For $d\ge 3$, the set $\Phi(d)$ is, at the moment of the writing of the paper, not known. Merel \cite{merel} proved that $\Phi(d)$ is finite for all positive integers $d$.

It is trivial to see that $\Phi(1)=\Phi^\infty(1)$ and $\Phi(2)=\Phi^\infty(2)$; this follows immediately from the fact that all the modular curves that parameterize the groups from $\Phi(1)$ are of genus 0 and that the modular curves that parameterize the groups from $\Phi(2)$ are of genus $\leq 2$. The set $\Phi^\infty(d)$ has been determined for $d=3$ by Jeon, Kim and Schweizer \cite{jks}; for $d=4$ by  Jeon, Kim and Park \cite{jkp-quartic}; and for $d=5$ and $6$, recently, by Derickx and Sutherland \cite{DS16}.

In \cite{naj} the second author determined $\Phi_\Q(2)$ and $\Phi_\Q(3)$. In this paper we determine $\Phi_\Q(p)$ for all primes $p\geq 7$ (see Corollary \ref{maincor}). Together with the paper \cite{GJ16} of the first author, which determines $\Phi_\Q(5)$, this completes the determination of all possible torsion groups over prime degree number fields of elliptic curves defined over $\Q$.

This paper is, as far as we are aware, the first to study the set $R_\Q(d)$. Lozano-Robledo \cite{loz} determined the sets $S_\Q(d)=\bigcup_{k\leq d}R_\Q(k)$ for $d\leq 42$. Note that the knowledge of $R_\Q(d)$ for all $d\leq B$ obviously gives $S_\Q(d)$ for all $d\leq B$, for any bound $B$, but not vice versa. Hence $R_\Q(d)$ can be considered to be a somewhat "finer" invariant. In section \ref{sec:RQD2} we determine $R_\Q(d)$ for all $d\leq 3342296$. This is done by applying the results of Section \ref{sec:RQD}.

In Section \ref{sec3} we use the results of Section \ref{sec:RQD} to show, in Theorem \ref{l-tors}, that $E(K)_{\tors}=E(\Q)_{\tors}$ for all number fields $K$ of degree $d$, where $d$ is an integer whose smallest prime divisor is $\geq 11$. In particular this applies for all number fields of prime degree $p$, for $p\geq 11$.

 In \cite{naj} it was shown that $\Phi_\Q(3)\not\subseteq \Phi^\infty(3)$, which implies $\Phi(3)\neq\Phi^\infty(3)$. {Derickx and van Hoeij \cite{Dvan14} proved that $\Phi(n)\neq \Phi^\infty(n)$ for $5\leq n \leq 17$}. Thus, it is natural to ask whether $\Phi(4)=\Phi^\infty(4)$? It was proved that $\Phi(3)\neq \Phi^\infty(3)$ by showing $\Phi_\Q(3)\not \subseteq \Phi^\infty(3)$, thus is natural to try to determine $\Phi_\Q(4)$, apart from it being interesting in its own right, for this purpose. This problem, of trying to determine $\Phi_\Q(4)$, originally motivated this work. We solve this problem completely in Section \ref{sec_quartic} and obtain that $\Phi_\Q(4)\subset\Phi^\infty(4)$.

Some of the proofs in the paper rely on extensive computations in \texttt{Magma} \cite{magma}. This primarily applies to computations in fixed finite groups. Where it is obvious that we are doing a computation in a fixed finite group, we will do it often without mention. All of the programs and calculation used for the proofs can be found at the research website of the first author.

\section{Notation and auxiliary results}\label{sec:auxiliary}

In this section we fix notation throughout the paper. If $E$ is an elliptic curve, then $j_E$ will denote its $j$-invariant. We will say that $E$ has CM or is a CM elliptic curve if it has complex multiplication. If it does not have complex multiplication, we say that the curve is non-CM.

There exists an $F$-rational cyclic isogeny $\phi:E\rightarrow E'$ of degree $n$ if and only if $\ker \phi$ is a $\Gal(\overline F / F)$-invariant cyclic group of order $n$; in this case we say that $E/F$ has an \textit{$n$-isogeny}.

In this section we will list, for the sake of the reader, some known results that we will be using throughout the paper.

The possible degrees of $n$-isogenies of elliptic curves over $\Q$ are known by the following theorem.

\begin{tm}[Mazur \cite{mazur2} and Kenku \cite{kenku2, kenku3, kenku4, kenku5}]
\label{tm-isog}
Let $E/ \Q$ be an elliptic curve with an $n$-isogeny over $\Q$. Then
\begin{equation}\label{isogenies}
n\in \{1,\ldots 19,21,25,27,37,43,67,163\}.
\end{equation}
There are infinitely many elliptic curves (up to $\overline{\Q}$-isomorphism) with an $n$-isogeny over $\Q$ for $n\in\{1,\ldots,10, 12,13,16,18,25\}$ and only finitely many for all the other $n$ listed in \eqref{isogenies}.

Furthermore, if $n\in \{14,19,27,43,67,163\}$ and $E$ has an $n$-isogeny over $\Q$, then $E$ has CM.
\end{tm}


Mazur \cite{maz} determined $\Phi(1)$:
$$\Phi(1)=\left\{ \cC_n\,:\, n=1,\ldots, 10, 12\right\} \cup \left\{\cC_2 \times \cC_{2n}| n=1,\ldots, 4\right\}.$$

Kamienny \cite{kam} and Kenku and Momose \cite{km} determined $\Phi(2)$:
$$\Phi(2)=\left\{ \cC_n\,:\, n=1,\ldots, 16, 18\right\} \cup \left\{\cC_2 \times \cC_{2n}| n=1,\ldots, 6\right\}\cup \left\{\cC_3 \times \cC_{3n}| n=1, 2\right\} \cup \left\{\cC_4 \times \cC_4\right\}.$$

The second author \cite{naj} determined $\Phi_\Q(2):$
$$\Phi_\Q(2)=\left\{ \cC_n\,:\,n=1,\ldots, 10, 12, 15, 16\right\} \cup \left\{\cC_2 \times \cC_{2n}\,:\, n=1,\ldots, 6\right\}.$$

\section{Images of mod $n$ Galois representations}\label{sec:galrep}

When studying the growth of torsion groups upon extensions, one is naturally led to the study of mod $n$ Galois representations attached to elliptic curves. Let $E/K$ be an elliptic curve over a number field $K$ and $n$ a positive integer. We denote by $E[n]$ the $n$-torsion subgroup of $E(\overline K)$, where $\overline K$ is a fixed algebraic closure of $K$. That is, $E[n]=\{P\in E(\overline K)\,|\, [n]P=\mathcal O\}$, where $\mathcal O$ denotes the identity of the group $E(\overline K)$. The field $K(E[n])$ is the number field obtained by adjoining all the $x$ and $y$-coordinates of the points of $E[n]$, or, equivalently, the smallest field over which all of $E[n]$ is defined. The absolute Galois group $\Gal(\overline K/K)$ acts on $E[n]$ by its action on the coordinates of the points, inducing a \textit{mod $n$ Galois representation attached to $E$}
$$
\rho_{E,n}\,:\,\Gal(\overline K/K)\longrightarrow \Aut(E[n]).
$$
Notice that since $E[n]$ is a free $\Z/n\Z$-module of rank $2$, fixing a basis $\{P,Q\}$ of $E[n]$, we identify $ \Aut(E[n])$ with $\GL_2(\Z/n\Z)$. Then we rewrite the above Galois representation as
$$
\rho_{E,n}\,:\,\Gal(\overline K /K)\longrightarrow\GL_2(\Z/n\Z).
$$
Therefore we can view $\rho_{E,n}(\Gal(\overline K /K))$ as a subgroup of $\GL_2(\Z/n\Z)$, determined uniquely up to conjugacy in $\GL_2(\Z /n\Z)$, and denoted by $G_E(n)$ from now on; it will always be clear what the field $K$ is. In certain instances we will say that $G_E(n)=G$ for some group $G$; this will always mean that we have fixed a basis so that we remove ambiguity of working up to conjugacy.

Since $K(E[n])=\{x,y\,|\, (x,y)\in E[n]\}$ is Galois over $K$ and since $\ker{\rho_{E,n}}=\Gal(\overline K/K(E[n]))$, we deduce that $G_E(n)\simeq\Gal(K(E[n])/K)$. Let $R=(x(R),y(R))\in E[n]$ and $K(R)=K(x(R),y(R))\subseteq K(E[n])$. Then by Galois theory there exists a subgroup $\mathcal H_R$ of  $\Gal(K(E[n])/K)$ such that $K(R)=K(E[n])^{\mathcal H_R}$. In particular, if we denote by $H_R$ the image of $\mathcal H_R$ in $\GL_2(\Z/n\Z)$, we have:
\begin{itemize}
\item
$
[K(R):K]=[G_E(n):H_R],
$
\item $\Gal(\widehat{K(R)}/K)\simeq G_E(n)/\mathcal N_{G_E(n)}(H_R)$, where $\widehat{K(R)}$ denotes the Galois closure of $K(R)$ over $K$, and $\mathcal N_{G_E(n)}(H_R)$ denotes the normal core of $H_R$ in $G_E(n)$.
\end{itemize}

We will use the following obvious lemma, sometimes without mention.

\begin{lemma}\label{obvious}
Let $E/K$ be an elliptic curve, $n$ a positive integer and $R \in E[n]$. Then $[K(R) : K]$ divides $|G_E(n)|$. In particular $[K(R) : K]$ divides $|\GL_2(\Z/n\Z)|$.
\end{lemma}

In practice, given the conjugacy class of $G_E(n)$ in $\GL_2(\Z/n\Z)$, we can deduce the relevant arithmetic properties of the fields of definition of the $n$-torsion points: since $E[n]$ is a free $\Z/n\Z$-module of rank $2$, we can identify the $n$-torsion points with $(a,b)\in (\Z/n\Z)^2$ (i.e. if $R\in E[n]$ and $\{P,Q\}$ is a $\Z/n\Z$-basis of  $E[n]$, then there exist $a,b\in \Z/n\Z$ such that $R=aP+bQ$). Therefore $H_R$ is the stabilizer of $(a,b)$ by the action of $G_E(n)$ on $(\Z/n\Z)^2$.

\

Let $p$ be an odd prime and $\epsilon=-1$ if $p\equiv 3\!\pmod{4}$ and otherwise let $\epsilon \geq 2$ be the smallest integer such that $\legendre{\epsilon}{p}=-1$.

Define the following matrices in $\GL_2(\F_p)$, for some $a,b\in \F_p$:
$$D(a,b) =\begin{pmatrix}
a & 0 \\
0 & b
\end{pmatrix}, \quad M_{\epsilon}(a,b) =\begin{pmatrix}
a & b\epsilon \\
b & a
\end{pmatrix}, \quad
T=\begin{pmatrix}0 & 1 \\1 & 0 \end{pmatrix} \quad\mbox{and}\quad J=\begin{pmatrix}1 &0  \\0 & -1 \end{pmatrix}.$$

We define the following subgroups of $\GL_2(\F_p)$:

\begin{align*}
  C_s(p) &=\left\{D(a,b) : a,b\in\F_p^\times   \right\}, \\
  C^+_s(p) & =\left\{D(a,b) ,\, T\cdot D(a,b) : a,b\in\F_p^\times \right\}, \\
  C_{ns}(p) &=\left\{M_{\epsilon}(a,b):(a,b) \in \F_p^2,\,\,(a,b)\ne (0,0)\right\},\\
  C^+_{ns}(p) &=\left\{M_{\epsilon}(a,b),\, J\cdot M_{\epsilon}(a,b):(a,b) \in \F_p^2,\,\,(a,b)\ne (0,0)\right\}.
\end{align*}
The group $C_s(p)$ is called the \textit{split Cartan subgroup} and $C_{ns}(p)$ is called the \textit{non-split Cartan subgroup}. Note that $C^+_s(p)$ (resp. $C^+_{ns}(p)$) is the normalizer of $C_{s}(p)$ (resp. $C_{ns}(p)$) in $\GL_2(\F_p)$. The orders of the groups $C_s(p)$, $C^+_s(p)$, $C_{ns}(p)$ and $C^+_{ns}(p)$ are $(p-1)^2$, $2(p-1)^2$, $p^2-1$ and $2(p^2-1)$, respectively.

\

\noindent{\textit{Notation.} }
Conjugacy classes of subgroups of $\GL_2(\F_p)$, for fixed small values of $p$, will be in this paper identified by labels introduced by Sutherland \cite[\S 6.4]{Sutherland2} and used in the LMFDB database \cite{lmfdb}. We should note that Zywina \cite{zyw} uses different notation for such conjugacy classes of subgroups of $\GL_2(\F_p)$. To aid the reader, in Tables \ref{tableSutherland} and \ref{tableSutherland2} we list both Sutherland's and Zywina's names for each of the groups.

We turn our attention to what is known on the possible images of mod $p$ Galois representations attached to elliptic curves defined over $\Q$, where $p$ is a prime. Determining all the possibilities of the images $G_E(p)\subseteq \GL_2(\F_p)$ for elliptic curves defined over $\Q$ is an important (and still open) problem in the theory of elliptic curves. The following theorem, due to Mazur, Serre, Bilu, Parent, Rebolledo, Zywina, Balakrishnan, Dogra, M\"uller, Tuitman and Vonk (see \cite{curse, bilu, pbr, serre1, zyw} for more details) gives the current state of the knowledge for non-CM elliptic curves.
\begin{tm}\label{tm_nonCM}
Let $E/\Q$ be a non-CM elliptic curve and $p$ a prime. Then one the following possibilities occurs:
\begin{romanenum}
\item $G_E(p)=\GL_2(\F_p)$.
\item $p\in  \{2, 3, 5, 7, 11, 13, 17, 37\}$, and $G_E(p)$ is conjugate in $\GL_2(\F_p)$ to one of the groups in Tables \ref{tableSutherland} and \ref{tableSutherland2}.
\item $p\geq 17$:
\begin{alphenum}
\item
$p \equiv 1\! \pmod{3}$, and $G_E(p)$ is conjugate in $\GL_2(\F_p)$ to $C^+_{ns}(p)$.
\item
$p \equiv 2\! \pmod{3}$, and $G_E(p)$ is conjugate in $\GL_2(\F_p)$ to $C^+_{ns}(p)$ or to the subgroup
$$
G_0(p):=\left\{M_{\epsilon}(a,b)^3,\, J\cdot M_{\epsilon}(a,b)^3\,:\, (a,b) \in \F_p^2,\,\,(a,b)\ne (0,0)\right\}\subseteq C^+_{ns}(p).
$$
\end{alphenum}
\end{romanenum}
\end{tm}

The complete and unconditional description of possible images $G_E(p)$ for $p\leq 11$ appearing in Theorem \ref{tm_nonCM}, which are compiled in Tables \ref{tableSutherland} and \ref{tableSutherland2} in the Appendix, is due to Zywina \cite{zyw} (see also Sutherland's paper \cite{Sutherland2}).

Before going any further we mention an important open conjecture, motivated by a question by Serre \cite{serre2} (see also \cite[Conjecture 1.12.]{zyw}), which conjecturally characterizes the possible images of Galois representations attached to non-CM elliptic curves.

\begin{conjecture}
If $E$ is a non-CM elliptic curve defined over $\Q$, $p\geq 17$ a prime and $(p, j_E)$ not in the set
\begin{equation}
\left\{(17,-17\cdot 373^3/2^{17}), (17,-17^2\cdot 101^3/2), (37,-7\cdot 11^3), (37, -7\cdot 137^3\cdot2083^3  \right \},
\end{equation}
then $G_E(p)=\GL_2(\F_p)$.
\label{serre_conj}
\end{conjecture}

\begin{remark}
If Conjecture \ref{serre_conj} is true, the case (iii) does not occur. All of the other case are known to occur.
\end{remark}

Some of our results will be best possible under the following Conjecture, which is a weaker version of \ref{serre_conj}.

\begin{conjecture}
\label{uniformity_conj}
If $E$ is a non-CM elliptic curve defined over $\Q$, $p\geq 17$ a prime and $(p, j_E)$ not in the set
\begin{equation}
\left\{(17,-17\cdot 373^3/2^{17}), (17,-17^2\cdot 101^3/2), (37,-7\cdot 11^3), (37, -7\cdot 137^3\cdot2083^3  \right \},
\end{equation}
then $G_E(p)=\GL_2(\F_p)$ or $C^+_{ns}(p)$.
\end{conjecture}

The difference between Conjectures \ref{serre_conj} and \ref{uniformity_conj} is that in the latter we allow that $G_E(p)= C^+_{ns}(p)$. For our purposes, there is no difference between $C^+_{ns}(p)$ and $\GL_2(\F_p)$, as both act transitively on the nonzero elements of $\F_p^2$.

If $E/\Q$ is a CM elliptic curve and $p$ a prime, the theory of complex multiplication gives us a lot of information about $G_E(p)$. In the CM case, the possibilities for $G_E(p)$ are completely understood. We list the possibilities for $G_E(p)$ (in the form stated in \cite[\S 1.9]{zyw}):

\begin{tm} \label{tm_CM}
Let $E/\Q$ be a CM elliptic curve and $p$ a prime.  The ring of endomorphisms of $E_{\Qbar}$ is an order of conductor $f$ in the ring of integers of an imaginary quadratic field of discriminant $-D$.
\begin{romanenum}
\item If $p=2$ then $G_E(2)=\GL_2(\F_2)$ or is conjugate in $\GL_2(\F_2)$ to \texttt{2B} or \texttt{2Cs}.
\item If $p>2$ and $(D,f)\ne (3,1)$:
\begin{alphenum}
\item If $\legendre{-D}{p}=1$, then $G_E(p)$ is conjugate in $\GL_2(\F_p)$ to $C^+_s(p)$.
\item If $\legendre{-D}{p}=-1$, then $G_E(p)$ is conjugate in $\GL_2(\F_p)$  to $C^+_{ns}(p)$.
\item Suppose that $p$ divides $D$ and hence $D=p$. Then $G_E(p)$ is conjugate in $\GL_2(\F_p)$  to one of the following subgroups
$$
G_{00}(p):=\left\{ \begin{pmatrix} a & b \\0 & \pm a \end{pmatrix} : a\in \F_p^\times, b\in \F_p\right\},
$$
$$
G_{10}(p): = \left\{ \begin{pmatrix} a & b \\0 & \pm a \end{pmatrix} : a\in (\F_p^\times)^2, b\in \F_p\right\}, \quad \text{ or }\quad G_{01}(p):= \left\{ \begin{pmatrix} \pm a & b \\0 &  a \end{pmatrix} : a\in (\F_p^\times)^2, b\in \F_p\right\}.
$$
\end{alphenum}
\item If $p>2$ and $(D,f)=(3,1)$:
\begin{alphenum}
\item  If $p \equiv 1 \!\pmod{9}$, then $G_E(p)$ is conjugate in $\GL_2(\F_p)$ to $C^+_{s}(p)$ .
\item If $p \equiv 8 \!\pmod{9}$, then $G_E(p)$ is conjugate in $\GL_2(\F_p)$ to $C^+_{ns}(p)$ .
\item If $p \equiv 4$ or $7 \!\pmod{9}$, then $G_E(p)$ is conjugate in $\GL_2(\F_p)$ to $C^+_{s}(p)$ or to the subgroup
$$
G^3(p):= \left\{D(a,ab^3) ,\, T\cdot D(a,ab^3) : a,b\in\F_p^\times \right\} \subseteq C^+_{s}(p).
$$
\item If $p \equiv 2$ or $5 \!\pmod{9}$, then $G_E(p)$ is conjugate in $\GL_2(\F_p)$ to $C^+_{ns}(p)$ or to the subgroup $G_0(p)$ of $C^+_{ns}(p)$.
\item  If $p=3$, then $G_E(3)$ is conjugate in $\GL_2(\F_3)$ to \texttt{3Cs.1.1}, \texttt{3Cs}, \texttt{3B.1.1}, \texttt{3B.1.2} or \texttt{3B}.
\end{alphenum}
\end{romanenum}
All the cases occur.
\end{tm}

\section{Growth of torsion in extensions}
\label{sec2}

In this section we prove general results, without any constraints on the base field over which the elliptic curve is defined, concerning the growth of the torsion of elliptic curves upon base change. We believe the results are primarily interesting in their own right, but they will also be useful in Section \ref{sec_quartic}.

Our first result gives restrictions on the growth of the $p$-torsion, in terms of the Galois group of the extension.

\begin{tm}
\label{maintm}
Let $L/F$ be a finite extension of number fields, $\widehat{L}$ the normal closure of $L$ over $F$, $G=\Gal(\widehat{L}/F)$, and suppose that  $H=\Gal(\widehat{L}/L)$ is a non-normal maximal subgroup of $G$. Let $p$ be a prime, $a=[F(\zeta_p): F]$,  and suppose $G$ does not contain a cyclic quotient group of order $a$. Then for every elliptic curve $E/F$, it holds that $E(L)[p]= E(F)[p]$.
\end{tm}
\begin{proof}
Suppose that $P\in E(L)$, $P \notin E(F)$ is a point of order $p$. Then, as $H$ is not normal in $G$ there exists a $\sigma \in G$ such that $\sigma(L) \neq L$. It follows that $P^\sigma \in (E(L))^\sigma=E(\sigma(L))$. Note that as $H$ is not contained in any proper subgroup of $G$, it follows that $L$ has no proper subfield containing $F$, so $L\cap L^\sigma=F$. As $P$ and $P^\sigma$ are defined over different fields, it follows that they are independent, i.e. $\langle P,P^\sigma\rangle=E[p]$. Obviously, $P,P^\sigma\in E(\widehat{L})$, so $E(\widehat{L})[p]\simeq \cC_p\times \cC_p$. But this is impossible, since $\widehat{L}$ does not contain $\zeta_p$, because $F(\zeta_p)$ would be a cyclic extension of $F$ of degree $a$, contradicting our assumptions.

\end{proof}

By specializing the above theorem, we immediately obtain the following corollary.

\begin{corollary}
\label{cor_p}
Let $E/ \Q$ be an elliptic curve and $L$ be a number field with Galois closure $\widehat{L}$ over $\Q$ and with $G:=\Gal(\widehat{L}/\Q)$. Suppose there are no intermediate fields $L\supseteq M\supseteq \Q$. Then:
\begin{romanenum}

  \item If $G$ is a simple group, and $p\geq 3$ be a prime, then $E(L)[p]=E(\Q)[p]$.

  \item If $G\simeq A_n$ for $n\geq 4$, and $p\geq 3$ be a prime, then $E(L)[p]=E(\Q)[p]$.

   \item If $G\simeq S_n$ for $n\geq 4$, and $p\geq 5$ be a prime, then $E(L)[p]=E(\Q)[p]$.

\end{romanenum}

\end{corollary}
\begin{proof}
Statement (i) obviously follows from Theorem \ref{maintm}, and statement (ii) follows from the fact that $A_n$ is simple for $n\geq 5$. It is easy to check that the conditions of Theorem \ref{maintm} are satisfied for $A_4$.

Statement (iii) follows for $n\geq 5$, from the fact that only proper normal subgroup of $S_n$ is $A_n$. For $n=4$, it is again easy to check that there are no cyclic quotients of order $\geq 4$ of $S_4.$
\end{proof}

We give another result that is useful for proving that over an extension with a certain Galois group there cannot be torsion growth of certain type.

\begin{tm}
\label{tmcrit2}
Let $E/F$ be an elliptic curve, $L/F$ be a finite extension of number fields with no intermediate fields, and let $G=\Gal(\widehat{L}/F)$, where $\widehat{L}$ is the normal closure of $L$ over $F$. If $G$ is not isomorphic to a quotient of $\Gal(F(E[p])/F)$, then $E(L)[p]= E(F)[p]$.
\end{tm}
\begin{proof}
Suppose the opposite, that $E(F)[p]\subsetneq E(L)[p]$. It follows, since $L\supseteq F$ has no intermediate fields, that $L\subseteq F(E[p])$. Since $F(E[p])$ is a Galois extension of $F$, it follows that $\widehat{L}\subseteq F(E[p])$, since $\widehat{L}$ is contained in any Galois extension of $F$ which contains $L$. It follows that $G$ is a quotient of $\Gal(F(E[p])/F)$.
\end{proof}

\begin{remark}
If $\Gal(F(E[p])/F)$ is a priori unknown, one can still use the fact that $\Gal(F(E[p])/F)$ is always a subgroup of $\GL_2(\F_p)$. Then, in the setting as above, if $G$ is not a quotient of any subgroup of $\GL_2(\F_p)$, then $E(L)[p]= E(F)[p]$.
\end{remark}

The previous two theorems give constraints on the growth of $E[p]$. One is naturally led to consider whether there can be any $p^{n+1}$ torsion over an extension $L$ of $F$ if $E(F)[p]\neq \{\mathcal O\}$. The following result will be useful.

\begin{tm}\label{tm5}
Let $L/F$ be a finite extension of number fields, $G=\Gal(\widehat{L}/F)$, where $\widehat{L}$ is the normal closure of $L$ over $F$, $n$ be a positive integer and let $p$ be a prime co-prime to $[L:F]$. Suppose that $G$ is not isomorphic to a quotient of any subgroup of $\GL_2(\Z/p^n\Z)$ and that $\Gal(\widehat{L}/L)$ is maximal in $G$. Let $E/F$ be an elliptic curve such that it has a $F$-rational point of order $p^n$, but no $F$-rational points of order $p^{n+1}$. Then $E(L)$ has no points of order $p^{n+1}$. \end{tm}
\begin{proof}
Suppose that there exist a point $P$ of order $p^{n+1}$ in $E(L)$. Since $\Gal(\widehat{L}/L)$ is maximal in $G$, it follows that
 $L\supseteq F$ has no intermediate fields, so $L$ is the field of definition of $P$, i.e. $L=F(P)$.

If $L\subseteq F(E[p^n])$ was true, then $\Gal(F(E[p^n])/L)$ would be a quotient of  $G_E(p^n)\simeq \Gal(F(E[p^n])/F)$, which is in turn a subgroup of $\GL_2(\Z /p^n \Z)$, contradicting our assumptions. We conclude that $L\nsubseteq F(E[p^n])$.

Note that since $G_E(p^{n+1})$ is a subgroup of the inverse image of reduction mod $p^n$ to $G_E(p^{n})$, it follows that $[F(E[p^{n+1}]): F(E[p^{n}])]$ divides $p^4$. Since $P\in F(E[p^{n+1}])$ and $P\notin F(E[p^{n}])$, it follows that $F(E[p^{n}]) \subsetneq F(E[p^{n+1}])$ We conclude that $[F(E[p^{n+1}]): F(E[p^{n}])]=p^k$, where $k=1,2,3$ or $4$.

Now since both $L=F(P)$ and $F(E[p^{n}])$ are contained in $F(E[p^{n+1}])$, it follows that
$$F(E[p^{n}])\subseteq LF(E[p^{n}]) \subseteq F(E[p^{n+1}]).$$
Since $L\nsubseteq F(E[p^n])$, we conclude that $[LF(E[p^{n}]): F(E[p^{n}])]$ divides $[F(E[p^{n+1}]): F(E[p^{n}])]$ and is hence equal to $p^s$ where $s=1,2,3$ or $4$. But since $[L:F]$ is coprime to $[F(E[p^{n}]):F]$, we have
$$[LF(E[p^{n}]): F(E[p^{n}])][F(E[p^{n}]):F]=[LF(E[p^{n}]): F]=[L:F][F(E[p^{n}]):F],$$
so $[L:F]=[LF(E[p^{n}]): F(E[p^{n}])]$, which obviously cannot be true since the left and right hand side are coprime and $>1$.
\end{proof}

\begin{proposition}
\label{prop_pdiv}
Let $E/F$ be an elliptic curve over a number field $F$, $n$ a positive integer, $P \in E(\overline F)$ be a point of order $p^{n+1}$. Then $[F(P):F(pP)]$ divides $p^2$ or $(p-1)p$.
\end{proposition}
\begin{proof}
Let $G=\Gal(F(E[p^{n+1}])/F(pP))$, and $H=\Gal(F(E[p^{n+1}])/F(P))$. By Galois theory, it follows that $[F(P):F(pP)]=[G:H]$.   To remove the ambiguity of having everything up to conjugacy in $\GL_2(\Z /p^{n+1}\Z)$, let $\{P,Q\}$ be a basis of $E[p^{n+1}]$ for some $Q\in E[p^{n+1}]$, let $G_E(p^{n+1})$ be induced by the action of $\Gal(\overline{F}/F)$ on this basis $\{P,Q\}$ and consider $G$ and $H$ to be subgroups of $G_E(p^{n+1})$. Then in particular, $G$ is a subgroup of
\begin{equation}\label{eqg}
G':=\left\{\begin{pmatrix}
                     a & b \\
                     c & d
                   \end{pmatrix}\in \GL_2(\Z/p^{n+1}\Z)\,:\, a-1\equiv c\equiv 0 \!\!\!\!\!\pmod {p^n}  \right\}.
\end{equation}
Moreover, $|G|$ divides $p^2\cdot p^n(p^n-p^{n-1})=p^{2n+1}(p-1)$.
Now let
$$\Gamma = \left\{\begin{pmatrix}
                     1 & x \\
                     0 & y
                   \end{pmatrix}\in \GL_2(\Z/p^{n+1}\Z)\right\};
                   $$
it follows that $H$ is isomorphic to $\Gamma \cap G$. Let $B$ be the kernel of reduction of $\GL_2(\Z/p^{n+1}\Z)$ mod $p^n$. We have $|B|=p^4$. Now we have that the kernel of reduction mod $p^n$ restricted to $G$ (resp. $H$) is $B\cap G$ (resp. $B\cap H$). Since $B\cap H$ and $B\cap G$ are subgroups of $B$, it follows that their order is a power of $p$. We see that $H$ (mod $p^n$) is a subgroup of $G$ (mod $p^n$) of index $d$ (where $d$ divides $|G\!\!\!\!\! \pmod{p^n}|)$, so
$$|G\!\!\!\!\! \pmod{p^n}|=d\cdot|H\!\!\!\!\! \pmod{p^n}| \quad \text{where }d|(p-1)p^{2n-1}.$$
So
$$|G|/|B \cap G|=|G\!\!\!\!\! \pmod{p^n}|=d\cdot|H\!\!\!\!\! \pmod{p^n}|=d\cdot|H|/|B \cap H|,$$
and
$$|G|/|H|=d\cdot|B \cap G|/|B \cap H|.$$
The right hand side of the equations can be divisible only by powers of $p$ and by divisors of $p-1$. From the definitions of $G$ and $H$, it follows that $[G:H]\leq p^2$, proving the proposition.
\end{proof}

\begin{remark}
\label{remark:maarten}
The results of Proposition \ref{prop_pdiv} are best possible, in the sense that there exist both $E/\Q$ and $P\in E(\overline \Q)$ such that $[\Q(P):\Q(pP)]=p^2$ and $[\Q(P):\Q(pP)]=p(p-1)$. The first case is "generic" and happens for any point of $p$-power order for any elliptic curve such that the $p$-adic representation attached to $E$ is surjective.

The case when $[\Q(P):\Q(pP)]=p(p-1)$ occurs for example when $G_E(p^2)$ is the reduction of $\Gamma_0(p^2)\cup \Gamma_1(p)$ mod $p^2$. Then $E(\Q)$ contains a point $Q$ of order $p$, and some of the solutions of $pP=Q$ are defined over a degree $p(p-1)$ extension (and some over a degree $p$ extension). One does not have to look far for an explicit example of such an elliptic curve - the elliptic curve \lmfdbec{11}{a}{3} in the LMFDB \cite{lmfdb} is such a curve for $p=5$.
\end{remark}

\begin{proposition}\label{2growth}
Let $E/F$ be an elliptic curve over a number field $F$, $n$ a positive integer, $P \in E(\overline F)$ be a point of order $2^{n+1}$ and let $\widehat{F(P)}$ be the Galois closure of $F(P)$ over $F(2P)$. Then $[F(P):F(2P)]$ divides $4$ and $\Gal(\widehat{F(P)}/F(2P))$ is either trivial, isomorphic to $\cC_2$, $\cC_2\times \cC_2$ or $D_4$.
\end{proposition}
\begin{proof}
Let $G=\Gal(F(E[2^{n+1}])/F(2P))$ and $H=\Gal(F(E[2^{n+1}])/F(P))$. Then by Galois theory, $N=\Gal(F(E[2^{n+1}])/\widehat{F(P)})$ is the normal core of $H$ in $G$ (i.e. the intersection of all conjugates of $H$ by elements of $G$) and
$$\Gal(\widehat{F(P)}/F(2P))\simeq  G/N.$$
 To remove the ambiguity of having everything up to conjugacy in $\GL_2(\F_p)$, let $\{P,Q\}$ be a basis of $E[2^{n+1}]$ for some $Q\in E[2^{n+1}]$, let $G_E(2^{n+1})$ induced by the action of $\Gal(\overline{F}/F)$ on the basis $\{P,Q\}$ and consider $G$ and $H$ to be subgroups of $G_E(2^{n+1})$. We have that $G$ is a subgroup of the group
\begin{equation}\label{eqg}
G'=\left\{\begin{pmatrix}
                     a & b \\
                     c & d
                   \end{pmatrix}\in \GL_2(\Z/2^{n+1}\Z)\,:\, a-1\equiv c\equiv 0\!\!\!\!\! \pmod {2^n}  \right\}.
\end{equation}
Define
$$
\Gamma= \left\{\begin{pmatrix}
                     1 & x \\
                     0 & y
                   \end{pmatrix}\in \GL_2(\Z/2^{n+1}\Z)\right\}.
                   $$
Then $H$ is isomorphic to $G\cap\Gamma$. We see that $H$ is a subgroup of $G$ whose index divides 4, so we obtain that $[F(P):F(2P)]$ divides $4$.

If $F(P)=F(2P)$, then obviously $\widehat{F(P)}=F(P)=F(2P)$, and the statement of the proposition is true.

If $[F(P):F(2P)]=2$, then obviously $F(P)$ is Galois over $F(2P)$ and $\Gal(\widehat{F(P)}/F(2P))\simeq \cC_2$.

If $[F(P):F(2P)]=4$, then it follows that
$$G/H=\left\{ H, \begin{pmatrix}
                  1 & 0 \\
                  2^n & 1
                \end{pmatrix}H,
                \begin{pmatrix}
                  1+2^n & 0 \\
                  2^n & 1
                \end{pmatrix}H,
                \begin{pmatrix}
                  1 & 0 \\
                  2^n & 1
                \end{pmatrix}H
                \right\}.$$
As $N$ is the intersection of the elements of $G/H$, by direct calculation we obtain that
$$
N=H \cap \left\{\begin{pmatrix}
                     a & b \\
                     c & d
                   \end{pmatrix}\in \GL_2(\Z/2^{n+1}\Z)\,:\, a-1\equiv c\equiv 0\!\!\!\!\! \pmod {2^{n+1}}, b \equiv 0\!\!\!\!\! \pmod 2  \right\}.
$$
There are now 2 possibilities: either $N=H$ or $[H:N]=2$. If $N=H$, then hence $F(P)$ is Galois over $F(2P)$, and we obtain that $G/N\simeq \cC_2\times \cC_2$.  On the other hand, if $[H:N]=2$, then $G/N\simeq D_4$.
\end{proof}

\section{Degree of the field of definition of $p$-torsion points on elliptic curves over $\Q$}
\label{sec:RQD}

In this section, for any prime $p$ we determine all the possible degrees $[\Q(P):\Q]$ for $P$ a point of order $p$ on an elliptic curve $E/\Q$.

\

Let $E/\Q$ be an elliptic curve and $p$ a prime. Then  $G_E(p)\subseteq \GL_2(\F_p)$ acts on $E[p]\simeq \F_p^2$. The set $\F_p^2$ can be written as the disjoint union of the orbits of $\F_p^2$ by $G_E(p)$, and by the Orbit-Stabilizer Theorem we obtain
\begin{equation}\label{eq2}
\# \F_p^2=p^2=\sum_{i=1}^k{[\Q(R_i):\Q]}=\sum_{i=1}^k{|G_E(p).(a_i,b_i)|},\qquad \mbox{where $R_i=a_iP+b_iQ$},
\end{equation}
$(a_1,b_1),\dots,(a_k,b_k)$ are representatives of the orbits of $\F_p^2$ by $G_E(p)$, and $G_E(p).v$ denotes the orbit of $v\in \F_p^2$ by the action of $G_E(p)$.

\

In order to determine all the possible degrees of the fields of definition of the $p$-torsion points of $E$, by equation (\ref{eq2}), we only need to know $G_E(p)\subseteq \GL_2(\F_p)$ and the cardinality of the orbits of $\F_p^2$ by $G_E(p)$.

 If we want to determine all the possible degrees of the extensions $\Q(P)$ over $\Q$ for points $P\in E(\overline \Q)$ of order $p$ for any elliptic curve $E/\Q$ we must, for all possible $G_E(p)$ and all $v\in\F_p^2$, $v\ne (0,0)$, determine the index $[G_E(p):G_E(p)_v]$ (where $G_E(p)_v$ is the stabilizer of $v$ by the action of $G_E(p)$ on $\F_p^2$), or equivalently $|G_E(p).v|$. Thanks to Theorems \ref{tm_nonCM} and \ref{tm_CM} we know a set that contains all the possibilities for $G_E(p)$ for any elliptic curve $E/\Q$ and any prime $p$. It will turn out that this set is close enough to the truth for us to get a complete unconditional list of all possible degrees $[\Q(P):\Q]$ for a large set of primes $p$.

\begin{lemma} \cite[Theorem 5.1]{loz}
\label{lem:divisibility}
Let $E/K$ be an elliptic curve over a number field and $p$ a prime such that $G_E(p)=\GL_2(\F_p)$. Then for a point $P$ of order $p$, we have $[\Q(P):\Q]=p^2-1$.
\end{lemma}

\begin{lemma}
\label{lem:divisibility2}
Let $E/K$ be an elliptic curve over a number field and $p$ a prime such that $G_E(p)=C_{ns}^+(p)$. Then for a point $P$ of order $p$, we have $[\Q(P):\Q]=p^2-1$.
\end{lemma}
\begin{proof}
In \cite[Lemma 7.5 (2)]{loz} it is claimed that $[\Q(P):\Q]=p^2-1$ or $2(p^2-1)$. But the case $[\Q(P):\Q]=2(p^2-1)$ is not possible since it is impossible that $[\Q(P):\Q]>(p^2-1)$, for any $P \in E[p]$ and any $G_E(p)$: This follows from the fact that $|G_E(p).v|$ can be at most $p^2-1$. The claim of the Lemma follows.
\end{proof}

\begin{lemma}\cite[Lemma 7.5. (1)]{loz}
\label{lem:divisibility3}
Let $E/K$ be an elliptic curve over a number field and $p$ a prime such that $G_E(p)$ is a conjugate of $C_{s}^+(p)$ in $\GL_2(\F_p)$. Then for a point $P$ of order $p$, we have $[\Q(P):\Q]$ is either $(p-1)^2$ or $2(p-1)$; both cases occur.
\end{lemma}

 We will need a slightly finer description of the possibilities for $G_E(13)$ than the one supplied in Theorem \ref{tm_nonCM}.

\begin{proposition}\label{prop13}
Let $E/\Q$ be a non-CM elliptic curve. Then $G_E(13)$ is conjugate in $\GL_2(\F_{13})$ to one of the groups in Table \ref{tableSutherland2}, to $C^+_s(13)$, $C^+_{ns}(13)$ or one of the following groups:
\begin{equation}\label{eq13}
\texttt{13Ns.2.1},\qquad \texttt{13Ns.5.2},\qquad \texttt{13Ns.5.1.4}
\end{equation}
\end{proposition}

\begin{proof}
First recall that every group in Table \ref{tableSutherland2} occurs. By Theorem \ref{tm_nonCM} (iii), the remaining possibilities are that $G_E(13)$ is a subgroup of $C^+_s(13)$, $C^+_{ns}(13)$ or \texttt{13S4}. We say that $G$ is an applicable subgroup of $\GL_2(\F_{13})$ if $-I \in G$, $\det(G)=\F_{13}^\times$ and $G$ contains an element with trace $0$ and determinant $-1$. Thanks to \cite[Proposition 2.2]{zyw} we have that if $G_E(13)\ne \GL_2(\F_{13})$ then $\pm G_E(13)$ is an applicable subgroup of $\GL_2(\F_{13})$. Using \texttt{Magma} we \href{http://matematicas.uam.es/~enrique.gonzalez.jimenez/research/tables/growth/p13.txt}{\color{blue}determine} that $G_E(13)$ is conjugate in $\GL_2(\F_{13})$ to $C^+_s(13)$, $C^+_{ns}(13)$, \texttt{13S4} or one of the following groups:
$$
\begin{array}{c}
\texttt{13Cs}, \texttt{13Cs.1.1}, \texttt{13Cs.1.11}, \texttt{13Cs.1.3}, \texttt{13Cs.1.4}, \texttt{13Cs.1.6}, \texttt{13Cs.1.8}, \texttt{13Cs.12.1},
\texttt{13Cs.12.3}, \\
\texttt{13Cs.12.4}, \texttt{13Cs.3.1}, \texttt{13Cs.3.4}, \texttt{13Cs.4.1}, \texttt{13Cs.5.1}, \texttt{13Cs.5.4},  \texttt{13Ns.2.1}, \texttt{13Ns.5.1.4}, \texttt{13Ns.5.2}.
\end{array}
$$
Now, it is not possible that $G_E(13)$ is conjugate to \texttt{13Cs} or \texttt{13Cs.a.b}, for any integers $a,b$, since then $E$ would have two independent $13$-isogenies over $\Q$. Therefore, $E$ would be 13-isogenous over $\Q$ to an elliptic curve with a $169$-isogeny, whose existence contradicts with Theorem \ref{tm-isog}.
\end{proof}

\begin{lemma}\label{dQ23}
Let $E/\Q$ be a elliptic curve and $P$ a point of order $13$ in $E$. Then
$$
[\Q(P):\Q]\in \{3,4,6,12,24,39,48,52,72,78,96,144,156,168\}.
$$
All the possible degrees occur.
\end{lemma}

\begin{proof} The fact that all of the degrees above appear can be seen from Table \ref{tableSutherland2} and from the \href{http://matematicas.uam.es/~enrique.gonzalez.jimenez/research/tables/growth/th5_7CM.txt}{\color{blue} degrees} of $[\Q(P):\Q]$ obtained from CM elliptic curves. The fact that no other degrees occur follows from Proposition \ref{prop13} and from \href{http://matematicas.uam.es/~enrique.gonzalez.jimenez/research/tables/growth/p13.txt}{\color{blue}computing} the lengths of orbits of all the possible images $G_E(p)$.
\end{proof}

The next result gives the possible degrees $[\Q(P):\Q]$ where $P$ is a point of order a prime $p\ne 13$ on $E$ such that $G_E(p)$ does not appear in Tables \ref{tableSutherland} and  \ref{tableSutherland2}.

\begin{tm}\label{dQ}
For an elliptic curve $E/\Q$ with a given $G_E(p)$ in the left column, and varying through all $P\in E[p]$, the values of $[\Q(P):\Q]$ obtained will be exactly does that appear in the right column.
$$
\begin{array}{c}
\begin{array}{cc}
\begin{array}{|c|c|}
\hline
G_E(p)  & [\Q(P):\Q]\\
\hline
\GL_2(\F_p) & p^2-1 \\
\hline
C^+_{ns}(p) & p^2-1 \\
\hline
C^+_{s}(p) & 2(p-1),  (p-1)^2 \\
\hline
\multicolumn{2}{c}{}
 \end{array}
 &
\begin{array}{|c|c|}
\hline
G_E(p) & [\Q(P):\Q]\\
\hline
G_{00}(p)  & p-1, (p-1)p \\
\hline
G_{01}(p) & p-1 , (p-1)p/2 \\
\hline
G_{10}(p)  & (p-1)/2, (p-1)p \\
\hline
\multicolumn{2}{c}{p\in\{3,7,11,19,43,67,163\}}
 \end{array}
   \end{array}
   \\[2mm]
 \begin{array}{|c|c|}
\hline
G_E(p) & [\Q(P):\Q]\\
\hline
\begin{array}{c} G^3(p) \\ p\equiv 1 \!\!\!\!\pmod{3}\end{array} & 2(p-1),  (p-1)^2/3 ,  2(p-1)^2/3 \\
\hline
\begin{array}{c} G_0(p) \\ p\equiv 2 \!\!\!\!\pmod{3}\end{array} & (p^2-1)/3,   2(p^2-1)/3 \\
\hline
 \end{array}
 \end{array}
$$
\end{tm}
\begin{remark}
The groups $G_{00}(p)$, $G_{01}(p)$, $G_{10}(p)$, $G^3(p)$ and $G_0(p)$ are as defined in Theorems \ref{tm_nonCM} and \ref{tm_CM}.
\end{remark}
\begin{proof}
The first table follows from Lemmas \ref{lem:divisibility},  \ref{lem:divisibility2} and  \ref{lem:divisibility3}. The second table follows from a direct \href{http://matematicas.uam.es/~enrique.gonzalez.jimenez/research/tables/growth/th5_6.txt}{\color{blue} computation}, as one has to consider only a small number of finite group actions.

Now let $G_E(p)=G^3(p)$. 
Then it is enough to consider the lengths orbits of the generators of the $p+1$ subgroups of $\F_p^2$. Equivalently, it is enough to know the size of the stabilizers of each of these elements. The stabilizer of $(1,0)$ is by \cite[Lemma 6.6 (1)]{loz} the subgroup of $G^3(p)$ given by
$$
H_{(1,0)}=\left\{D(1,b):b\in \F_p^\times\right\}.
$$
Obviously, the stabilizer is of size $(p-1)/3$, hence the size of the orbit of $(1,0)$ is $\frac{|G^3(p)|}{(p-1)/3}=2(p-1)$. Using \cite[Lemma 6.6]{loz} and the same argumentation, we conclude that the length of the orbit of $(0,1)$ is also $2(p-1)$. The stabilizer of any element of the form $(1,x)$ for $x\in \F_p^\times$ in $C_s^+(p)$ is by \cite[Lemma 6.6 (3)]{loz}
$$
H_{(1,x)}=\left\{ D(1,1),D(x,x^{-1})\right \}.
$$
Thus the stabilizer of $(1,x)$ in $G^3(p)$ is $H\cap G^3(p)$; this set obviously contains either $1$ or $2$ elements. It contains $2$ elements, by the definition of $G^3(p)$, if and only if $x^2$ is a cube in $\F_p^\times$. Obviously there exist both $x$'s that satisfy this condition (take for example $x$ to be a cube itself) and $x$ that do not. Hence there exist vectors $(1,x)$ which are in an orbit of length $\frac{(p-1)^2}3$ and those that are in an orbit of length $\frac{2(p-1)^2}3$. In this case $\F_p^2$ decomposes into orbits of lengths $1,\frac{(p-1)^2}3,\frac{2(p-1)^2}3$.

Now consider the case $G_E(p)=G_0(p)$. It follows from \cite[Lemma 7.4]{loz}, that the stabilizer of each element is of order $1$ or $2$. If there exists an element whose stabilizer is of order 1, then $\F_p^2$ decomposes into orbits of length $1,\frac{p^2-1}{3},\frac{2(p^2-1)}{3}$, and if there does not exist such an element, then $\F_p^2$ decomposes in orbits of length $1,\frac{p^2-1}{3},\frac{p^2-1}{3},\frac{p^2-1}{3}$ under the action of $G_E(p)$.

By \cite[Lemma 7.4]{loz}, the elements of $C_{ns}^+(p)$ that fix any non-trivial vector (or equivalently have an eigenvalue 1) are the identity and the matrices $J\cdot M_{\epsilon}(a,b) $ such that $a^2-\epsilon b^2=1$, or equivalently those such that  $\det J\cdot M_{\epsilon}(a,b) =-1$. It is easy to see that there are $p+1$ such matrices in $C_{ns}^+(p)$, each fixing each element in a different subgroup of $(\F_p)^2$.

We now study the number of such elements in $G_0(p)$. Clearly, none of them are of the form $M_{\epsilon}(a,b)^3$, so all of them must be of the form $J\cdot M_{\epsilon}(a,b)^3$.

Now, $J\cdot M_{\epsilon}(a,b)^3$ has determinant $-1$ if and only if $M_{\epsilon}(a,b)^3$ has determinant 1. Since $1$ is the only third root of unity in $\F_p$, it follows that $\det (M_{\epsilon}(a,b)^3)=1$ if and only if $\det M_{\epsilon}(a,b)=1$.
The subgroup $\mathcal M$ of $C_{ns}(p)$ of matrices $M_{\epsilon}(a,b)$ of determinant 1 is the kernel of the determinant map $\det: C_{ns}(p)\rightarrow \F_p^\times$, which is surjective. Hence, by the first isomorphism theorem, $\# \mathcal M=p+1$. As $\#\mathcal M$ is divisible by 3, it follows that  $\mathcal M^3=\{M_{\epsilon}(a,b)^3:(a,b)\in(\F_p)^2,(a,b)\ne (0,0)\}$ has cardinality $\frac{p+1}{3}$, as claimed.

Hence there are  $\frac{p+1}{3}$ non-identity matrices that each fix every element in exactly one subgroup of $(\F_p)^2$ of order $p$. As there are $p+1$ subgroups of $(\F_p)^2$ of order $p$, it follows that there will be $\frac{2(p+1)}{3}$ subgroups such that each of their non-trivial elements has trivial stabilizer, proving the claim.
\end{proof}
We now state the main result of this secion:

\begin{tm}\label{exact_degree}
Let $E/\Q$ be an elliptic curve, p a prime and $P$ a point of order $p$ in $E$. Then all of the cases in the table below occur for $p\leq 13$ or $p=37$, and they are the only ones possible. The degrees in the table below with an asterisk occur only when $E$ has CM.
$$
\begin{array}{|c|c|}
\hline
p & [\Q(P):\Q]\\
\hline
2 & 1,2,3\\
\hline
3 & 1,2,3,4,6,8\\
\hline
5 & 1,2,4,5,8,10,16,20,24\\
\hline
7 & 1,2,3,6,7,9,12,14,18,21,24^*,36,42,48\\
\hline
11 & 5,10,20^*,{ 40}^*,55,{ 80}^*,100^*,110,120\\
\hline
13 & 3,4,6,12,{ 24}^*,39,{ 48}^*,52,72,78,96,{ 144}^*,156,168\\

\hline
37 & 12,36,{ 72}^*,444,{ 1296}^*,1332, 1368\\
\hline
\end{array}
$$
For all other $p$, the possibilities for $[\Q(P):\Q]$ are as is given below. The degrees in equations \eqref{eq8} - \eqref{eq10} occur only for CM elliptic curves $E/\Q$. Furthermore the degrees in equations \eqref{eq10} occur only for elliptic curves with $j$-invariant $0$. If Conjecture \ref{uniformity_conj} is true, then the degrees in equations \eqref{eq11} also occur only for elliptic curves with $j$-invariant $0$.
\begin{align}
p^2-1 \quad \quad \quad \quad \quad \quad \quad \quad \quad \quad & \text{for all $p$},  \\
8,16,32^*,136,256^*,272,288    \quad\quad \quad\quad \quad & \text{for }p=17 \\
(p-1)/2,\,\,p-1,\,\,p(p-1)/2,\,\,\,p(p-1) \quad \quad  & \text{if }p \in \{19,43,67,163\}, \label{eq8} \\
2(p-1),\,\,\,\,(p-1)^2 \quad \quad \quad \quad \quad \quad \quad   & \text{if }p\equiv 1\!\!\!\!\! \pmod{3}\text{ or } \left( \frac{-D}{p}\right )=1 \text{ for any }D\in\mathcal{CM},  \\
(p-1)^2/3,\,\,\,\,2(p-1)^2/3  \quad \quad \quad \quad \quad \quad & \text{if }p\equiv 4,7\!\!\!\!\! \pmod{9}, \label{eq10}\\
(p^2-1)/3,\,\,\,\,2(p^2-1)/3  \quad \quad \quad \quad \quad \quad & \text{if }p\equiv 2,5\!\!\!\!\! \pmod{9}, \label{eq11}
\end{align}
where $\mathcal{CM}=\{1,2,7,11,19,43,67,163\}$.

Apart from the cases above that have been proven to appear, the only other options that might be possible are:
\begin{equation}\label{ser_imp}
\begin{array}{ccl}
(p^2-1)/3,\,\,\,\,2(p^2-1)/3  & & \text{if }p\equiv 8\!\!\!\!\!  \pmod{9}.\\
\end{array}
\end{equation}
\end{tm}

\begin{proof}
The theorem follows directly from the results of this section, combined with Theorems \ref{tm_nonCM} and \ref{tm_CM}, and from direct \href{http://matematicas.uam.es/~enrique.gonzalez.jimenez/research/tables/growth/th5_7CM.txt}{\color{blue} calculations} of lengths of orbits of fixed finite groups.
\end{proof}

\begin{remark}
  Note that $3$ would be redundant in the set $\mathcal{CM}$, although there exists an elliptic curve with CM by an order of an imaginary quadratic field of discriminant $-3$, as $p\equiv 1 \!\pmod 3$ immediately implies $\left( \frac{-D}{p}\right )=-1$. That is why $3$  has been left out from the set $\mathcal{CM}$.
\end{remark}

\begin{remark}
If Conjecture \ref{uniformity_conj} is true, then the cases \eqref{ser_imp} do not occur.
\end{remark}

\section{Points of prime order on elliptic curves $E/\Q$ over number fields of degree $d$}
\label{sec:RQD2}
Lozano-Robledo \cite{loz} denotes by $S_\Q(d)$ the set of primes $p$ for which there exists a number field $K$ of degree $\le d$ and an elliptic curve $E/\Q$ such that the $p$ divides the order of $E(K)_{\tors}$. He determined the sets $S_\Q(d)$ for $d\le 42$. To more easily state his results, denote by $S^*_{\Q}(d)=S_{\Q}(d)\smallsetminus S_{\Q}(d-1)$, then the table below shows the necessary data to recover $S_\Q(d)$ for $d\le 42$:
$$
\begin{array}{|c|c|c|c|c|c|c|c|c|c|}
\hline
d & 1 &3 &  5 &  8&  9&  12& 21 & 33 &\text{ other } d\leq 42\\
\hline
S^*_{\Q}(d) & 2,3,5,7 & 13 & 11 & 17 & 19 & 37 & 43 & 67 & \emptyset\\
\hline
\end{array}
$$
In this section, we obtain results about a set that will give us more information:  $R_\Q(d)$. Recall from the introduction that we denote by $R_\Q(d)$ the set of primes $p$ for which there exists a number field $K$ of degree exactly $d$ and an elliptic curve $E/\Q$ such that the $p$ divides the order of $E(K)_{\tors}$. Obviously, knowing $R_\Q(d)$ immediately gives $S_\Q(d)$, but not vice versa. We obtain our results as more or less direct consequences of Theorem \ref{exact_degree}.

The first obvious observation is that for all positive integers $n,d$, if $p\in R_\Q(d)$, then $p\in R_\Q(nd)$. Obviously, we have $2,3,5,7 \in R_\Q(d)$ for all positive integers $d$.

\
The following corollary summarizes the results that we obtain.

\begin{corollary}\label{RQD}
Let $\mathcal{CM}=\{1,2,7,11,19,43,67,163\}$. The following holds:
\begin{romanenum}
  \item $11 \in R_\Q(d)$ if and only if $5|d$.
  \item $13 \in R_\Q(d)$ if and only if $3|d$ or $4|d$.
  \item $17 \in R_\Q(d)$ if and only if $8|d$.
  \item $37 \in R_\Q(d)$ if and only if $12|d$.
  \item For $p=19, 43, 67, 163$, it holds that $p\in R_\Q(d)$ if and only if $\frac{p-1}{2}|d$.
  \item For $p\geq 23,$ $p\neq  37,43, 67, 163$:
  \begin{alphenum}
          \item If $p\equiv 1 \!\pmod 3$, then it holds that $p\in R_\Q(d)$ if and only if $2(p-1)|d$.
          \item If $p\equiv 2 \!\pmod 3$ and $\left(\frac{-D}{p}\right)=1$ for at least one of the values $D\in \mathcal{CM}$ then $p\in R_\Q(d)$ if and only if $2(p-1)|d$.
          \item If $p\equiv 2,5 \!\pmod 9$ and $\left(\frac{-D}{p}\right)=-1$ for all the values $D\in \mathcal{CM}$ then $p\in R_\Q(d)$ if and only if $(p^2-1)/3|d$.
          \item Let $p\equiv 8 \!\pmod 9$ and $\left(\frac{-D}{p}\right)=-1$ for all the values $D \in \mathcal{CM}$. Then if $p\in R_\Q(d)$, then it follows that $\frac{p^2-1}3\mid d$. On the other hand if $p^2-1|d$, then it follows that $p\in R_\Q(d)$.
  \end{alphenum}
\end{romanenum}
\end{corollary}
\begin{proof}
This follows directly from Theorem \ref{exact_degree} and direct computations.
\end{proof}

\begin{remark}
Note that the only case where we do not have one necessary and sufficient condition for $p$ to be in $R_\Q(d)$ is (vi) (d). This is because of the possibility that $G_E(p)=G_0(p)$ in that case. If Conjecture \ref{uniformity_conj} is true, then $G_E(p)=G_0(p)$ will be impossible for all $p>13$, and we will have $p\in R_\Q(d)$ if and only if $p^2-1|d$ in this case.

Of course, it is a natural question to ask what is the density of primes satisfying the assumptions of (vi) (d).  Notice that $ \left(\frac{-4}{p}\right)=\left(\frac{-1}{p}\right)=-1$ is true if and only if $p\equiv 3 \!\pmod 4$. The condition $\left(\frac{-2}{p}\right)=-1$ is then equivalent to $\left(\frac{2}{p}\right)=1$, which implies that $p\equiv 7 \!\pmod 8$.

Now, by quadratic reciprocity, $$\left(\frac{-D}{p}\right)=\left(\frac{-1}{p}\right)\left(\frac{D}{p}\right)=(-1)\cdot(-1)\cdot\left(\frac{p}{D}\right)=\left(\frac{p}{D}\right),$$ for all $D\in \mathcal{CM}-\{1,2\}$. For each of these $D$, the condition $\left(\frac{p}{D}\right)=-1$ is satisfied by half of the residue classes modulo $D$. Combining, by the Chinese remainder theorem, all these conditions together with $p\equiv 7 \!\pmod 8$ and $p\equiv 8 \!\pmod 9$ we obtain that the assumptions of (vi) (d) are satisfied by $\frac{1}{4}\cdot \frac{1}{6} \cdot\left( \frac{1}{2}\right)^6=\frac{1}{1536}$ of the residue classes modulo $8\cdot 9 \cdot 7\cdot11\cdot19\cdot43\cdot67\cdot 163$. As the primes are evenly distributed in residue classes, we conclude that a set of natural density $\frac{1}{1536}$ of primes satisfies the assumptions of (vi) (d).
\end{remark}

\begin{remark}
Let $p$ be a prime and denote by $\mathcal D_p$ the minimal subset of the possible degrees $[\Q(P):\Q]$ for any point $P$ of order $p$ in any elliptic curve $E/\Q$, such that $[\Q(P):\Q]$ is divisible by some $d\in \mathcal D_p$. Then
$$
R_\Q(d)=\{\mbox{$p$ prime}\,:\, \exists\, a \in \mathcal D_p\,\,\mbox{such that}\,\, a|d\}.
$$
\end{remark}

\begin{corollary}
Let $p=2$ or $p>5$ be a prime. Then $R_{\Q}(p)=R_\Q(1)=\{2,3,5,7\}$.
\end{corollary}

In this context we denote by
$$
R^*_{\Q}(d)=R_{\Q}(d)\smallsetminus \bigcup_{
\substack{d'|d \\ d'\neq d}
} R_{\Q}(d')
$$

The table below shows the \href{http://matematicas.uam.es/~enrique.gonzalez.jimenez/research/tables/growth/RQ.txt}{\color{blue}data} needed to recover $R_\Q(d)$ for $d\le 100$:
$$
\begin{array}{|c|c|c|c|c|c|c|c|c|c|c|c|c|c|c|c|}
\hline
d & 1 &3,4 &  5 &  8&  9&  12& 21 & 33 & 44 & 56 & 60 & 80 & 81 &92 &\text{ other } d\leq 100\\
\hline
R^*_{\Q}(d) & 2,3,5,7 & 13 & 11 & 17 & 19 & 37 & 43 & 67 & 23 & 29 & 31 & 41 & 163 & 47 & \emptyset \\
\hline
\end{array}
$$
Note that the smallest prime that satisfies the condition of Corollary \ref{RQD} (vi)(d) is $p=3167$. Therefore, the smallest $d$ where we do not have an unconditional determination of $R_\Q(d)$ is $d=3343296$.

\section{Torsion groups of elliptic curves $E/\Q$ over number fields of degree $d$}
\label{sec3}

In this section we want to say things about the possibilities for the whole torsion group of elliptic curves $E/\Q$ over number fields, not just about points of prime degree, as in the previous section.

We write $E(K)[p^\infty]$ to denote the $p$-primary torsion subgroup of $E(K)_{\tors}$, that is, the $p$-Sylow subgroup of $E(K)$.

We will prove results about elliptic curves in prime degree fields. We first need to study the case of degree $7$ separately.

Recall from the introduction that $\Phi_{\Q}(d)\subseteq \Phi(d)$ is the set of possible isomorphism classes of groups $E(K)_{\tors}$, where $K$ runs through all number fields $K$ of degree $d$ and $E$ runs through all elliptic curves defined over $\Q$.

\begin{proposition}
	\label{7-torsion}
    $\Phi_\Q(7)=\Phi(1)$.
	\end{proposition}
\begin{proof}  Let throughout the proof $E/\Q$ be an elliptic curve and $K$ a number field of degree $7$. First suppose that $E$ has CM. We can see from \cite[Theorem 1.4]{clark-TAMS} (or \cite[\S 4.7]{clark}) that in the CM case, no subgroups occur in degree $7$ which do not occur over $\Q$.

Now assume $E/\Q$ is non-CM. First we show that $E(K)[p^\infty]=E(\Q)[p^\infty]$ for $p\neq 7$ prime. For $p\leq 11$, it follows that $E(K)[p^\infty]=E(\Q)[p^\infty]$ from Table \ref{tableSutherland} in the Appendix (note that by \cite{zyw}, this list of mod $p$ Galois images for $p\leq 11$ is complete). For $p\geq 13$ it follows directly from Corollary \ref{RQD}.

It remains to deal with the case $p=7$. From Table \ref{tableSutherland} we see that $E(\Q)[7]= E(K)[7]$ unless $G_E(7)$ is conjugate in $\GL_2(\F_7)$ to \texttt{7B.1.3}. We claim that if $G_E(7)$ is conjugate in $\GL_2(\F_7)$ to \texttt{7B.1.3}, then $E(\Q)[p]=\{\mathcal O\}$ for all $p\neq 7$. Indeed, such an elliptic curve would have a $7p$-isogeny, and this is possible only for $p=2$ and $p=3$, and for only finitely many elliptic curves, up to $\Qbar$-isomorphism. Note that if $j_E\notin\{0,1728\}$ and $E'$ is a quadratic twist of $E$, then $\pm G_E(n)=\pm G_{E'}(n)$ for any integer $n\ge 3$.  We explicitly \href{http://matematicas.uam.es/~enrique.gonzalez.jimenez/research/tables/growth/prop7_1a.txt}{\color{blue} check} that no elliptic curve with a $14$-isogeny or a $21$-isogeny has $G_E(7)$ conjugate in $\GL_2(\F_7)$ to \texttt{7B.1.3}.

It remains to prove that $E(K)_{\tors}$ contains no points of order 49.  Looking at Table \ref{tableSutherland} we observe that a necessary condition for $E(K)_{\tors}$ to contain a point of order $7$ is that $G_E(7)$ is conjugate in $\GL_2(\F_7)$ to \texttt{7B.1.1} or \texttt{7B.1.3}. Assume that there exists a number field $K$ of degree $7$ such that $E(K)[49]=\langle P \rangle \simeq \cC_{49}$. Then $G_E(49)$ satisfies:
$$
G_E(49)\equiv\, G_E(7) \,\,(\mbox{mod $7$})\qquad\mbox{and}\qquad [G_E(49):H_P]=7,
$$
where $H_P$ is the image in $\GL_2(\Z/49\Z)$ of the subgroup $\mathcal H_P\subseteq \Gal(\Q(E[49])/\Q)$ such that $\Q(P)=\Q(E[49])^{\mathcal H_P}$.

Note that in general we do not have an explicit description of possible images $G_E(49)$, but using \texttt{Magma} we can run through all subgroups of $\GL_2(\Z/49\Z)$ that reduce to $G_E(7)$.

First assume that $G_E(7)$ is conjugate in $\GL_2(\F_7)$ to \texttt{7B.1.3}. We \href{http://matematicas.uam.es/~enrique.gonzalez.jimenez/research/tables/growth/prop7_1b.txt}{\color{blue} check} using \texttt{Magma} that for any subgroup $G$ of $\GL_2(\Z/49\Z)$ satisfying $G \equiv H\,\, (\mbox{mod $7$})$ for some conjugate $H$ of \texttt{7B.1.3} in $\GL_2(\Z/7\Z)$, and for any $v\in (\Z/49\Z)^2$ of order $49$, we have $[G:G_v]\ne 7$, where $G_v$ is the stabilizer of $v$ by the action of $G$ on $(\Z/49\Z)^2$. Therefore for any point $P\in E(\overline \Q)$ of order 49 we have $[G_E(49):H_P]\ne 7$. In particular this proves that if $G_E(7)$ is conjugate in $\GL_2(\F_7)$ to \texttt{7B.1.3}, then $E(K)_{\tors}$ contains no points of order 49.

Now assume that $G_E(7)$ is conjugate in $\GL_2(\F_7)$ to \texttt{7B.1.1}. Using a similar procedure as the one used before, we \href{http://matematicas.uam.es/~enrique.gonzalez.jimenez/research/tables/growth/prop7_1b.txt}{\color{blue}search} for subgroups $G$ of $\GL_2(\Z/49\Z)$ satisfying $G \equiv H\,\, (\mbox{mod $7$})$ for some subgroup $H$ conjugate to \texttt{7B.1.1} in $\GL_2(\Z/7\Z)$, and satisfying $[G:G_v]= 7$ for some $v\in (\Z/49\Z)^2$ of order $49$. There exist such groups $G$, but in all the cases we obtain that $G_v$ is a normal subgroup of $G$ and $G/G_v\simeq \cC_7$. Therefore we have deduced that if $E/\Q$ is an elliptic curve such that $G_E(7)$ is conjugate in $\GL_2(\F_7)$ to \texttt{7B.1.1} and there exists a number field $K$ of degree $7$ with a $K$-rational point of order $49$, then $K$ is Galois. But this would imply that $E$ has a rational $49$-isogeny, and that is impossible by Theorem \ref{tm-isog}.

This finishes the proof.

\end{proof}

\begin{tm}	\label{l-tors}
Let $B$ be a positive integer. Let $E/\Q$ be an elliptic curve and $K/\Q$ a number field of degree $d$, where the smallest prime divisor of $d$ is $\geq B$. Then
\begin{romanenum}
  \item If $B\ge 11$, then $E(K)[p^\infty]=E(\Q)[p^\infty]$ for all primes $p$. In particular, $E(K)_{\tors}=E(\Q)_{\tors}$.
  \item If $B\ge 7$, then $E(K)[p^\infty]=E(\Q)[p^\infty]$ for all primes $p\neq 7$.
  \item If $B\ge 5$, then $E(K)[p^\infty]=E(\Q)[p^\infty]$ for all primes $p\neq 5,7,11$.
  \item If $B>2$, then $E(K)[p^\infty]=E(\Q)[p^\infty]$ for all primes $p\neq 2,3,5,7,11,13,19,43,67,163$.
\end{romanenum}
\end{tm}

\begin{proof}
We will prove only the case (i) - the cases (ii), (iii) and (iv) are proved using the exact same argumentation. Suppose first that $P\in E(\overline \Q)$ is a point of order $p\leq 13$. Then directly (and obviously) from Theorem \ref{dQ} we see that $[\Q(P):\Q]$ has to be divisible by some element from the set $\{2,3,5,7\}$.

For $p\geq 17$, by looking at the possibilities listed in Theorem \ref{dQ} we have that $[\Q(P):\Q]$ always attains only obviously even values, except in the cases
$$[\Q(P):\Q]=\frac{p-1}{2}\qquad \text{ and }\qquad [\Q(P):\Q]=\frac{p(p-1)}{2},$$
which are possible only for $p\in\{3,7,11,19,43,67,163\}$. Since the cases $p\in\{3,7,11\}$ have already been dealt with, it remains to look at only the cases $p\in\{19,43,67,163\}$. Since in all of these cases we have $p\equiv 1 \!\pmod 3$, it follows that $\frac{p-1}{2}$ will be divisible by $3$. We conclude that $E(K)[p]=E(\Q)[p]$ for all primes $p$ if $K$ satisfies the assumptions of the theorem. Now it follows from Proposition \ref{prop_pdiv} that $E(K)[p^\infty]=E(\Q)[p^\infty]$.
\end{proof}

We immediately obtain the following result.

\begin{corollary}\label{maincor}
Let $d$ be a positive integer such that the smallest prime factor of $d$ is $\geq 11$. Then $\Phi_\Q(d)=\Phi(1)$.
\end{corollary}

\begin{remark}
One could show, using arguments as in the proof of Proposition \ref{prop_pdiv}, that  for any $p\in\{2,3,5,7\}$,
the set
$$\bigcup_{n=1}^{\infty}\Phi_\Q(p^n)$$
will be infinite, containing $\cC_{p^k}$, for all positive integers $k$.

So, in a sense, Corollary \ref{maincor} is best possible.
\end{remark}

\begin{remark}\label{abbey}
From Corollary \ref{maincor} and the previous remark it follows that the $d$ such that $E(K)_{\tors}=E(\Q)_{\tors}$ for all elliptic curves $E/\Q$ and all number fields $K$ of degree $d$ are exactly those such that the smallest prime divisor of $d$ is $\geq 11$. These are $d$ of the form $d=210k+x$, where $k$ is a nonnegative integer and $1 \leq x < 210$ is an integer coprime to 210. Thus it follows that exactly $\frac{\phi(210)}{210}=\frac{8}{35}$
of all integers (when ordered by size) satisfy this property.
\end{remark}

Our methods give results not only about the torsion in finite extensions, but also in infinite extensions of $\Q$.

\begin{corollary}
\label{cor:zp}
Let $p\geq 11$ be a prime and let $K$ be the $\Z_p$-extension of $\Q$. Then $E(K)_{\tors}=E(\Q)_{\tors}$.
\end{corollary}
\begin{proof}
   Let $\ell$ be a prime. First suppose that $E(\Q)[\ell]=\{\mathcal O\}$. Then we have $E(F)[\ell]=\{\mathcal O\}=E(\Q)_{\tors}$ by Theorem \ref{l-tors} for all number fields $F$ contained in $K$.

   Now suppose that $E(\Q)[\ell]\neq \{\mathcal O\}$.  Then it follows that $\ell\leq 7$, and hence $p\neq \ell$. If $\ell=2$, then it is obvious that $E(\Q)[2]=E(K)[2]$. If $3\leq \ell \leq 7$, then $E(\Q)[\ell]=E(K)[\ell]$, since $K$ does not contain $\ell$-th roots of unity.

   By Proposition \ref{prop_pdiv}, we now have $E(\Q)[\ell^\infty]=E(F)[\ell^\infty]$ for all number fields $F$ contained in $K$, from which we again conclude that  $E(\Q)[\ell^\infty]=E(K)[\ell^\infty]$.
\end{proof}

\subsection{Torsion growth in degree $7$ extensions}
In this section we consider how the torsion grows in prime degree extensions, in the following sense: for each possible $E(\Q)_{\tors}$, what can $E(K)_{\tors}$ possibly be, for $K/\Q$ an extension of prime degree $d$. There is no torsion growth in degree $p\geq 11$ by Theorem \ref{l-tors} and the torsion growth for $p\leq 5$ has been completely described in \cite{GJT15} for $p=2$, in \cite{GJNT15} for $p=3$ and in \cite{GJ16} for $p=5$. Thus it remains to consider the case $p=7$. We prove the following.

\begin{proposition}
Let $E/\Q$ be an elliptic curve and $K$ a number field of degree $7$.
\begin{romanenum}
\item If $E(\Q)_{\tors}\not\simeq  \cC_1$, then $E(\Q)_{\tors}=E(K)_{\tors}$.
\item If  $E(\Q)_{\tors}\simeq  \cC_1$, then $E(K)_{\tors}\simeq \cC_1$ or $\cC_7$. Furthermore, if $E(\Q)_{\tors}\simeq  \cC_1$ and $E(K)_{\tors}\simeq \cC_7$, then $K$ is the unique degree $7$ number field with this property and $E$ is isomorphic to the elliptic curve
$$
\begin{array}{rcl}
E_t\,:\,y^2&=&x^3+27(t^2-t+1)(t^6+229t^5+270t^4-1695t^3+1430t^2-235t+1)x \\
& &\quad  +54(t^{12}-522t^{11}-8955t^{10}+37950t^9-70998t^8+131562t^7\\
& &\qquad  -253239t^6+316290t^5-218058t^4+80090t^3-14631t^2+510t+1).
\end{array}
$$
for some value $t\in \Q$.
\end{romanenum}
\end{proposition}

\begin{proof}
Let $E/\Q$ be an elliptic curve. Clark, Corn,  Rice, Stankewicz \cite[\S 4.7]{clark} have proved that if $E$ has CM, then $E(K)_{\tors}$ has torsion points only of order $1, 2,3,4$ or $6$. Now from Lemma \ref{obvious}, it follows that $E(K)[n]=E(\Q)[n]$ for $n=2,3,4,6$. It follows that $E(K)_{\tors}=E(\Q)_{\tors}$.

Then we can assume that $E/\Q$ is non-CM. In the proof of Proposition \ref{7-torsion} we proved that if $p$ is a prime and $K/\Q$ is a number field of degree $7$ then $E(\Q)[p^\infty]=E(K)[p^\infty]$, unless $G_E(7)$ is conjugate in $\GL_2(\F_7)$ to \texttt{7B.1.3}, in which case $E(\Q)[p]=\{\mathcal O\}$  for all $p\ne 7$. Then suppose that $G_E(7)$ is conjugate in $\GL_2(\F_7)$ to \texttt{7B.1.3}. Zywina \cite[Theorem 1.5 (iv)]{zyw} proved that, under these assumptions, there exists a $t_0\in\Q$ such that $E$ is $\Q$-isomorphic to $E_{t_0}$.

Let $\psi_7(x)$ be the $7$-division polynomial of $E_t$ and $\psi_7(x)=f_1(x)f_2(x)f_3(x)f_4(x)$, its irreducible factorization (over $\Q(t)$), where $f_1(x)$ is of degree $3$ and $f_i(x)$, for $i=2,3,4$, are all of degree $7$. For all specializations $E_{t_0}$, $t_0 \in \Q$, we have that $G_{E_{t_0}}(7)$ is conjugate in $\GL_2(\F_7)$ to \texttt{7B.1.3}, as it cannot be any smaller, so the $7$-division polynomial of $E_{t_0}$ will factor over $\Q$ in the same manner as the 7-division polynomial of $E_t$ factors over $\Q(t)$. We \href{http://matematicas.uam.es/~enrique.gonzalez.jimenez/research/tables/growth/prop7_7.txt}{\color{blue} compute} that the roots of $f_i(x)$, for $i=2,3,4$, define the same degree 7 extensions of $\Q(t)$, up to isomorphism. We obtain that $E_t$ has torsion $\cC_7$ over this extension, completing the proof.
\end{proof}

{
\begin{remark}
Let $E/\Q$ be an elliptic curve and $p$ a prime. If the possible degrees of points of order $p$ are $p$ and $p-1$, then $\Q(E[p])=K\Q(\zeta_p)$ satisfies $[\Q(E[p]):\Q]=p(p-1)$ and there exists $P\in E[p]$ such that $K=\Q(P)$ is a Kummer extension, i.e. there exists $s\in \mathbb Z$, a non $p^\text{th}$-power, such that $K=\Q(\sqrt[p]{s})$. This is the case when $G_E(p)$ is conjugate in $\GL_2(\F_p)$ to \texttt{3B.1.2}, \texttt{5B.1.2} or \texttt{7B.1.3}.
\end{remark}
}

\section{Torsion of elliptic curves $E /\Q$ over quartic number fields}
\label{sec_quartic}

In this section we solve the problem of determining the set of $\Phi_\Q(4)$. The sets $\Phi_\Q(2)$ and $\Phi_\Q(3)$ have been determined by the second author \cite{naj}, $\Phi_\Q(5)$ by the first author \cite{GJ16}.

Let $G$ be a finite group; define $\Phi^{G}_\Q(n)$ to be the set of isomorphism classes of groups $E(K)_{\tors}$, where $E$ is an elliptic curve defined over $\Q$ and $\Gal(\widehat{K}/\Q)\simeq G$, where $\widehat{K}$ denotes the normal closure of $K$ over $\Q$. The first steps taken towards determining $\Phi_\Q(4)$ have been done by Chou \cite{chou}, who determined $\Phi^{\operatorname{C}_4}_\Q(4)$ and $\Phi^{\operatorname{V}_4}_\Q(4)$, i.e he found the possible torsion groups over Galois quartic number fields of elliptic curves defined over $\Q$.

The strategy used both in \cite{chou} and \cite{naj} (for determining $\Phi_\Q(3)$) is to view how $\Gal(\widehat{K}/\Q)$ acts on $E(K)_{\tors}$ and to conclude that this forces certain properties on $E$ over $\Q$. For certain choices of $E(K)_{\tors}$, one would get properties of $E$ over $\Q$ that would violate known results about elliptic curves over $\Q$, thus ruling out the possibility of that particular torsion group $E(K)_{\tors}$. This approach works well when $K$ is Galois over $\Q$, as in \cite{chou}, or "not far from Galois" in \cite{naj} where $\Gal(\widehat{K}/\Q)$ was either $\cC_3$ or $S_3$ (even in this case, $S_3$ posed some problems). We make use of the machinery developed in Section \ref{sec2} to show obstructions for possible growth from $E(\Q)_{\tors}$ to $E(K)_{\tors}$ when $K$ is "far away from being Galois" (meaning informally that $[\widehat{K}:K]$ is relatively large).

The results of Section \ref{sec2} immediately tell us that (see Corollary \ref{cor_p}) $$\Phi^{\operatorname{A}_4}_\Q(4)=\Phi(1).$$

In some instances in the remainder of the paper we will say that a quartic field $K$ has Galois group $D_4,\ A_4$ or $ S_4$, which will always mean that its normal closure $\widehat{K}$ satisfies $\Gal(\widehat{K}/\Q)\simeq D_4,\ A_4$ or $ S_4$, respectively.

We will need to look at the case of the possible $3$-torsion growth in a quartic number field with Galois group $S_4$ separately, as this case is a bit more delicate.

\begin{proposition}
\label{prop_s4_3}
Let $E/\Q$ be an elliptic curve and $K$ a quartic number field with Galois group isomorphic to $S_4$. Then $E(K)[3]=E(\Q)[3]$.
\end{proposition}
\begin{proof}
Suppose that $E(\Q)[3]\subsetneq E(K)[3]$. Then by Theorem \ref{tmcrit2}, we have that $S_4$ is isomorphic to a quotient of a subgroup of $G_E(3)$. We deduce that $G_E(3)=\GL_2(\F_3)$, since $\#S_4=24$, $\#\GL_2(\F_3)=48$, and $\GL_2(\F_3)$ \href{http://matematicas.uam.es/~enrique.gonzalez.jimenez/research/tables/growth/prop8_1.txt}{\color{blue}has} no subgroups isomorphic to $S_4$. But then the number field of smallest degree over which there exists a point $P$ of order $3$ satisfies
$$\Gal(\Q(E[3])/\Q(P)) \simeq \left\{ \begin{pmatrix}
            1 & c \\
            0 & d
          \end{pmatrix}\,:\, c \in \F_3, d\in \F_3^\times\right\},
$$
so by Galois theory is a number field of degree 8 over $\Q$.
\end{proof}

\begin{remark}
We should note that if $G_E(3)= \GL_2(\F_3)$, as in the proof above, then $\Q(E[3])$ in fact contains a Galois extension $K$ of $\Q$ with Galois group isomorphic to $S_4$. This will be the field generated by all the $x$-coordinates of $E[3]$. However, it will hold that $E(K)[3]=\{\mathcal O\}$.
\end{remark}

Now we obtain the following result.

\begin{corollary}
\label{cor_pn}
Let $p\geq 3$ be a prime, $n$ a positive integer and $K$  a quartic number field with Galois group isomorphic to $A_4$ or $S_4$. If an elliptic curve $E/\Q$ has no points of order $p^n$ in E$(\Q)$, then it has no points of order $p^n$ in $E(K)$.
\end{corollary}
\begin{proof}
For $n\geq 2$, the result follows from Theorem \ref{tm5}. Suppose from now on that $n=1$. If $\Gal(\widehat{K}/\Q)\simeq S_4$ and $p=3$ (resp. $p\ge 5$) the results follows from Proposition \ref{prop_s4_3} (resp. Theorem \ref{tmcrit2}). If $\Gal(\widehat{K}/\Q)\simeq A_4$ the result follows from Theorem \ref{tmcrit2} for any $p\ge 3$.
\end{proof}

As a corollary of Proposition \ref{2growth}, we can now obtain the following results about subsets of $\Phi_\Q(4)$.

\begin{corollary}\label{cor_triv}
$\Phi^{\operatorname{S}_4}_\Q(4)  = \Phi^{\operatorname{A}_4}_\Q(4)  = \Phi(1)$.
\end{corollary}
\begin{proof}
Suppose $K$ is a quartic number field with Galois group $A_4$ or $S_4$.

Let $E/\Q$ be an elliptic curve. Because of Corollary \ref{cor_p} and Proposition \ref{prop_s4_3}, we have $E(\Q)[p]=E(K)[p]$, for all $p\geq 3$. Due to Corollary \ref{cor_pn}, we have that there are no points of order $p^n$ over $K$, for $p\geq 3$ and $n\geq 2$.

For $p=2$, we note that $\Q(E[2])$ is either $\Q$, a quadratic field, cyclic cubic or an $S_3$ extension of $\Q$ and the intersection of $K$ with any of these fields is $\Q$, so $E(\Q)[2]=E(K)[2]$. Furthermore, by Proposition \ref{2growth}, if $E(K)[2^n]\supsetneq E(\Q)[2^n]$, for $n\geq 2$, then the Galois group of $K$ is isomorphic to a subgroup of $D_4$, which is a contradiction.
\end{proof}

Thus to determine $\Phi_\Q(4)$ completely (see Corollary \ref{cor_Q4}), it remains to just determine $\Phi^{\operatorname{D}_4}_\Q(4)$.

\subsection{Determining $\Phi^{\operatorname{D}_4}_\Q(4)$}\label{sec_D4}

\smallskip

In this section we determine $\Phi^{\operatorname{D}_4}_\Q(4)$. In particular, we prove the following theorem.

\begin{tm}\label{tm_D4}
The set $\Phi^{\operatorname{D}_4}_\Q(4)$ is given by
$$
\begin{array}{rcl}
\Phi^{\operatorname{D}_4}_\Q(4) &=&
\left\{ \cC_n \; : \; n=1,\dots,10,12,15,16,20,24 \right\} \cup
	\left\{ \cC_2 \times \cC_{2m} \; : \; m=1,\dots,6 \right\} \\	
& &	\qquad \cup \left\{ \cC_3 \times \cC_{3m} \; : \; m=1,2 \right\} \cup  \left\{ \cC_4 \times \cC_{4m} \; : \; m=1,2 \right\}.
\end{array}
$$
\end{tm}

\begin{remark}
$\Phi^{\operatorname{D}_4}_\Q(4)= \Phi_\Q(2) \cup\left\{  \cC_{20}\,,\,\cC_{24} \right\}.$
\end{remark}

This will complete the determination of $\Phi_\Q(4)$.

\begin{corollary}
	\label{cor_Q4}
	The set $\Phi_\Q(4)$ is given by
	\begin{eqnarray*}
		\Phi_{\mathbb Q}(4)\!\!\!& = &\!\!\!\!\!\left\{ \cC_n \; : \; n=1,\dots,10,12,13,15,16,20,24 \right\} \cup \left\{ \cC_2 \times \cC_{2m} \; : \; m=1,\dots,6,8 \right\} \cup \\
		& &  \left\{ \cC_3 \times \cC_{3m} \; : \; m=1,2 \right\} \cup  \left\{ \cC_4 \times \cC_{4m} \; : \; m=1,2 \right\} \cup  \left\{ \cC_5 \times \cC_{5}  \right\} \cup  \left\{ \cC_6 \times \cC_{6}  \right\}.
	\end{eqnarray*}
\end{corollary}

\begin{remark}\label{rem}
 Corollary \ref{cor_Q4} implies that the list of possible pairs $(E(\Q)_{\tors}$, $E(K)_{\tors})$, where $K$ is a quartic field, obtained by the first author and Lozano-Robledo in \cite{GL16} is complete.
\end{remark}

Throughout this section, denote by $K$ a quartic number field whose normal closure $\widehat{K}$ over $\Q$ has Galois group isomorphic to $D_4$. We will say that such a quartic number field is \textit{dihedral quartic number field}.  Denote by $F$ the unique quadratic field contained in $K$. Elliptic curves will always, unless stated otherwise, be defined over $\Q$.

We will make use of the following lemma.

\begin{lemma}
\label{lem-isog}
Let $p$ be a prime, $E/F$ an elliptic curve over a number field $F$ and let $L/F$ be a finite extension such that $\Gal(\widehat {L}/F)$ does not contain as quotient a cyclic group of order $a=[F(\zeta_p): F]$, where $\widehat{L}$ is the normal closure of $L$ over $F$.

Then if $E(L)$ has a point $P$ of order $p$, it follows that $E$ has a $p$-isogeny over $F$. Furthermore  $F(P)$ is Galois over $F$ and $\Gal(F(P)/F)$ is a cyclic group of order dividing $p-1$.
\end{lemma}
\begin{proof}
Suppose $E(L)$ contains a point $P$ of order $p$. By our assumptions, $\widehat {L}$ does not contain $\Q(\zeta_p)$, since $\Gal(F(\zeta_p)/F)\simeq \cC_a$ would be a quotient of  $\Gal(\widehat {L}/F)$, contradicting our assumptions. Hence $E(\widehat L)\simeq \cC_p$, and since $\widehat L$ is Galois over $F$, it follows that $\diam{P}$ is a cyclic Galois-invariant subgroup, and hence $E$ has a $p$-isogeny $\phi$ over $F$ such that $\diam{P} =\ker \phi$.  Since $P$ is in the kernel of a $F$-rational isogeny, it follows that $F(P)$ is Galois over $F$ of degree dividing $p-1$, over $F$.
\end{proof}

\subsubsection{$p$-primary torsion subgroup}

The first step in determining $\Phi^{\operatorname{D}_4}_\Q(4)$ is to determine the possibilities for $E(K)[p^\infty]$ for all primes $p$.

\begin{lemma}\label{prop_p}
If $P$ is a point of prime order $p$ in $E(K)$, then $p\in\{2,3,5,7\}$.
\end{lemma}
\begin{proof}

First note that $R_\Q(2)=\{2,3,5,7\}$, therefore the primes in the above set appear as orders of points over dihedral quartic number fields. By Corollary \ref{RQD}, we have $R_\Q(4)=\{2,3,5,7, 13\}$

We claim that there are no points of order 13 over dihedral quartic number fields. Suppose the opposite; let $P\in E(K)$ be a point of order 13. By Lemma \ref{lem-isog}, $\Q(P)$ is Galois over $\Q$, and hence $\Q(P)=\Q$ or $\Q(P)=F$. But that is in contradiction with the fact that $R_\Q(2)=\{2,3,5,7\}$.
\end{proof}

\begin{lemma}\label{prop_2}
 The order of $E(K)[p^\infty]$ divides $16,9,5,7$, for $p=2,3,5,7$ respectively. In particular all the possible orders occur.
\end{lemma}

\begin{proof}
Let $n$ be a positive integer such that $n_p\le 4,2,1,1$, if $p=2,3,5,7$ respectively, then $\cC_{p^{n_p}}\in \Phi(1)\subseteq\Phi_\Q^{D_4}(4)$.

We split the proof depending on the prime $p$:

$\bullet$ $p=2$: The possible torsion subgroups of elliptic curve with CM defined over a quartic number field have been determined in \cite{clark}. In particular, it is proven that the order of the $2$-primary torsion subgroup divides $16$. Therefore we can assume $E$ is non--CM. The minimal degree of definition of any $2$-subgroup of $E(\Qbar)$ has been determined in \cite[Theorem 1.4]{GL15} in the non--CM case. In particular, for degree $4$ we have (see Table 1 in \cite{GL15}) that if $E/\Q$ is a non-CM elliptic curve and $K/\Q$ is a quartic number field then
$$
E(K)[2^\infty]\in \{\,\cC_1\,,\, \cC_2\,,\,\cC_4\,,\,\cC_8\,,\,\cC_{16}\,,\,\cC_2\times\cC_2\,,\,\cC_2\times\cC_4\,,\,\cC_2\times\cC_8\,,\,\cC_2\times\cC_{16}\,,\,\cC_4\times\cC_4\,,\,\cC_4\times\cC_8\,\}.
$$
We will prove that the groups $\cC_2\times\cC_{16}$ and $\cC_4\times\cC_8$ do not occur over dihedral quartic number fields.

Rouse and Zureick-Brown \cite{rouse} have classified all the possible $2$-adic images of $\rho_{E,2^\infty}:\Gal(\Qbar/\Q)
\to \GL_2(\Z_2)$, and have given explicitly all the $1208$ possibilities. The first author and Lozano-Robledo \cite{GL16} have determined for each possible image the degree of the field of definition of any $2$-subgroup. From the file \texttt{2primary\_Ss.txt} (this file can be found at the research website of the first author) one can read out if a given $2$-subgroup is defined over a number field of given degree $d$. In particular, the $2$-adic images up to conjugation (following the notation of \cite{rouse}) for the case $d=4$ and the given $2$-subgroups are:
\begin{itemize}
\item $\cC_2\times\cC_{16}$:  \texttt{X193j, X193l, X193n, X213i, X213k, X215c, X215d, X235f, X235l, X235m},
\item $\cC_4\times\cC_8$: \texttt{X183d, X183g, X183i, X187d, X187h, X187j, X187k, X189b, X189d, X189e, X193g,
X193i, X193n, X194g, X194h, X194k, X194l, X195h, X195j, X195l}.
\end{itemize}
 For each of these images, we \href{http://matematicas.uam.es/~enrique.gonzalez.jimenez/research/tables/growth/lem8_11.txt}{\color{blue}determine} the Galois group of the quartic number field over $\Q$ corresponding to each of the above $2$-adic images (i.e. the possible Galois groups are the normal cores of index $4$ stabilizers of the appropriate element in $(\Z/2^n\Z)^2$). We obtain that all the quartic number fields are Galois over $\Q$. This finishes the proof for the $2$-primary torsion.

$\bullet$ $p=3$: We have $\cC_3\times\cC_3\in \Phi_\Q(2)$, therefore it occurs as the torsion subgroup of elliptic curves over dihedral quartic number fields. By \cite[Theorem 7 (11)]{GL16} we have that $\cC_3\times\cC_9$ is not a subgroup of $E(K)$. Thus, to finish determining the possible $3$-primary torsion it is enough to prove that there is no $27$-torsion over a dihedral quartic number field $K$. Suppose the opposite, i.e $\cC_{27}$ is a subgroup of $E(K)$. Since $F$ is a subfield of $K$ and $[K:F]=2$, there exists a $\delta\in F$ such that $K=F(\sqrt{\delta})$. Then we have
$$
E(K)[27]\simeq E(F)[27]\oplus E^\delta(F)[27].
$$
Therefore $E(F)[27]$ or $E^\delta(F)[27]$ contains $\cC_{27}$. But this is not possible since $\cC_{27}$ is not a subgroup of any group in $\Phi(2)$.

$\bullet$ $p=5$:  It is not possible to have the full $5$-torsion over a dihedral quartic number field $K$, since otherwise $\Q(\zeta_5)=K$, and this is impossible since $\Q(\zeta_5)$ is cyclic. We observe that $\cC_{25}$ is not possible over dihedral quartic number fields by using the same argument as in the $\cC_{27}$ case.

$\bullet$ $p=7$: It is impossible that $E[7]\subseteq E(K)$ since $K\not\supseteq \Q(\zeta_7)$. To prove that there is no $49$-torsion, we again use the same argument as in the $\cC_{27}$ case.

\end{proof}

\subsubsection{Combinations of $p$-primary torsion}
In the previous section we have determined the $p$-primary torsion of an elliptic curve $E/\Q$ over a dihedral quartic number field. In this section we study the possible combinations of $p$-power torsion and $q$-power torsion, for primes $p\neq q$. First we prove that if $E(K)$ has a point of order $7$, then $E(K)[p]=\{\mathcal O\}$ for all primes $p\neq 7$.

\begin{lemma}
\label{lem7}
If $7$ divides the order of $E(K)_{\tors}$, then $E(K)_{\tors}\simeq \cC_7$.
\end{lemma}

\begin{proof}
Suppose  $E(K)_{\tors}$ has a point of order $7$. Let $p$ be a prime, $p\ne 7$ such that $p$ divides the order of $E(K)_{\tors}$. By Lemma \ref{prop_p} we have that $p\in\{2,3,5\}$. The case $p=2$ is not possible, since there cannot be any $14$-torsion over a quartic number field (the proof of \cite[Proposition 3.9]{chou} can be slightly modified to prove it). The case $p=3$ is also not possible, since there cannot exist points of order $21$ over a quartic number field (see \cite[Theorem 7 (iv)]{GL16}). Finally, suppose $p=5$. Then we have $E(\widehat{K})[35]\simeq \cC_{35}$, because $\widehat{K}$ does not contain $\Q(\zeta_5)$ or $\Q(\zeta_7)$, as $\cC_4$ and $\cC_6$ are not quotients of $D_4$. This implies that $E$ has a $35$-isogeny over $\Q$ which is not possible by Theorem \ref{tm-isog}.

Now, by Lemma \ref{prop_2}, we deduce that $E(K)_{\tors}\simeq \cC_7$.
\end{proof}

\begin{lemma}
If $15$ divides the order of $E(K)_{\tors}$, then $E(F)_{\tors}\simeq \cC_{15}$.
\end{lemma}

\begin{proof}
We have that $\cC_{15}\in\Phi_\Q(2)$, therefore the case $E(K)_{\tors}\simeq \cC_{15}$ is possible. Now, let $P\in E(K)_{\tors}$ be a point of order $15$. Theorem 8 \cite{GL16} shows that $\Q(P)$ is an abelian extension. Thus, since $K$ is a dihedral quartic number field, we have that $\Q(P)=F$, where  $F$ is the unique quadratic subfield of $K$.

If there existed a point $R$ of order 30 in $E(K)_{\tors}$, then $2R$ would be a point of order $15$ and $\Q(2R)=F$ as above. The point $R$ is not defined over $F$, as $\cC_{15}\notin\Phi_\Q(2)$. But then $[\Q(R):\Q(2R)]=3$, which is a contradiction. We conclude that $E(K)_{\tors}$ cannot contain a point of order 30.

Now by Lemma \ref{prop_2} we have that the $3$-primary torsion is isomorphic to one of the groups: $\cC_3$, $\cC_9$, $\cC_3\times \cC_3$. To finish the proof we prove that $9$ does not divide the order of $E(K)_{\tors}$. If the $3$-primary torsion is isomorphic to $\cC_9$, we have that any point of order $9$ is defined over a proper subfield of $K$ (see Theorem 7 (10) from \cite{GL16}). By the previous argument we deduce that $E$ has a point of order $45$ over the quadratic field $F$, which is not possible. Finally, let us suppose that the $\cC_{15}\times \cC_3$ is a subgroup of $E(K)_{\tors}$. Then the quadratic subfield of $K$ is $\Q(\sqrt{-3})$, and we claim that both the subgroups isomorphic to $\cC_5$ and to $\cC_3\times \cC_3$ have to be defined over $\Q(\sqrt{-3})$, which is clearly impossible. For the subgroup $\cC_5$, this follows from Lemma \ref{lem-isog}, while for $\cC_3\times \cC_3$ it follows from the fact that $\Q(E[3])$ is Galois over $\Q$.

\end{proof}

We will later make use of the following lemma.

\begin{lemma}
\label{lem2x2n}
  Let $k$ be any field, $E/k$ an elliptic curve and suppose $E(k)$ contains a subgroup isomorphic to $\cC_{4n}$. Then $E$ is $2$-isogenous over $k$ to an elliptic curve $E'$ such that $E'(k)$ contains a subgroup isomorphic to $\cC_2 \times \cC_{2n}$.
\end{lemma}
\begin{proof}
Let $A$ be the subgroup of $E(k)$ isomorphic to $\cC_{4n}$, let $B$ its order 2 subgroup and let $C$ be its order $4$ subgroup. We now see that $E$ is 2-isogenous to $E'=E/B$ and $E'$ is also $2$-isogenous to $E''=E/C$. Thus $E'$ has 2 independent $2$-isogenies, and hence full 2-torsion over $k$. Let $\phi:E\rightarrow E'$ be the $2$-isogeny obtained by quotienting out by $B$. Then it follows that $\phi(A)\simeq \cC_{2n}\subseteq E'(k)$, proving the lemma.
\end{proof}

\begin{lemma}
If $10$ divides the order of $E(K)_{\tors}$, then $E(K)_{\tors}$ is isomorphic to either to $\cC_{10}$, $\cC_{20}$ or to $\cC_2 \times \cC_{10}$.
\end{lemma}
\begin{proof}
By Lemma \ref{prop_2} we need to prove that $E(K)_{\tors}$ cannot contain $\cC_2 \times \cC_{20}$ or $\cC_{40}$. But by Lemma \ref{lem2x2n}, it is enough to prove that $\cC_2 \times \cC_{20}$ is impossible.

Suppose then that $E(K)_{\tors}\supseteq \cC_2 \times \cC_{20}$. By Lemma \ref{lem-isog}, we have that $E$ has a $5$-isogeny over $\Q$ and that $\Q(P)$ is Galois over $\Q$. Thus it follows that $\Q(P)=F$ or $\Q(P)=\Q$. If $\Q(P)=F=\Q(\sqrt d)$, then we have $E(F)[5]=E(\Q)[5]\oplus E^d(\Q)[5]$. Then either $E(\Q)[5] \simeq \cC_5$ or $E^d(\Q)[5]\simeq \cC_5$. In the latter case we can study $E^d$ instead of $E$; thus it follows that we can take without loss of generality that $\Q(P)=\Q$.

 We conclude that $E(\Q)_{\tors}\simeq \cC_{10}.$ Let $Q\in E(\Q)$ be a point of order $10$. The torsion of $E$ will grow from $\cC_{10}$ in $\Q$ to $\cC_2\times \cC_{20}$, if and only if a solution $R$ of the equation $2R=Q$ satisfies that $\Q(R)$ contains the quadratic field $\Q(\sqrt{\Delta_E})=\Q(E[2])$ as a subfield. The universal elliptic curve over $X_1(10)$ (which can be thought of as  the parametric family of elliptic curves with $\cC_{10}$ torsion) is (see \cite{kub}):
$$
E_t\,:\,y^2 + (t^3 - 2t^2 - 4t + 4)xy + (t-1)(t-2)(t^2-6t+4)t^3y = x^3 -(t-1)(t-2)t^3x^2,
$$
where we can take $Q=(0,0)\in E_t[10]$.
Note that the four solutions of $2R=Q$ are defined over at most 2 number fields, but they will be conjugate to each other, and hence the following argumentation does not depend on the choice of $R$. The condition that $\Q(\sqrt{\Delta_E})$ is contained in $\Q(R)$ is, after some calculation, equivalent to the hyperelliptic curve
$$
C:y^2 = x(x-1)(x-2)(x^2-x-1)(x^2 - 6 x +4)
$$
having a solution such that $y\neq 0$. To find all the rational points on $C$, we note that $\Aut(C)\simeq D_4$ and for a non-hyperelliptic involution $\sigma \in \Aut(C)$ we \href{http://matematicas.uam.es/~enrique.gonzalez.jimenez/research/tables/growth/lem8_15.txt}{\color{blue}obtain} that the quotient curve $X=C/\diam{\sigma}$ is of genus $1$, where $X$ is defined by the following equation
$$X:y^2 = x^3 - 7x - 6.$$
We have $X(\Q)=\{\mathcal O, (-1,0), (-2,0), (3,0)\}\simeq \cC_2 \times \cC_2$ and the map $\phi: C\rightarrow X$ sends the Weierstrass points of $C$ to the $2$-torsion points of $X$. We \href{http://matematicas.uam.es/~enrique.gonzalez.jimenez/research/tables/growth/lem8_15.txt}{\color{blue}obtain} that the 4 rational Weierstrass points on $C$, $\{\infty, (0,0), (1,0), ( 2,0)\}$ map to $\{\mathcal O, (-1,0)\}$; the preimages of the 2 remaining points in $X(\Q)$ are the 4 quadratic Weierstrass points of $C$. This proves that
$$C(\Q)=\left \{\infty, (0,0), (1,0), \left(2,0\right)\right\}.$$
We have proved that $E(K)$ cannot contain $\cC_2 \times \cC_{20}$, proving the lemma.
\end{proof}

\begin{lemma}\label{6div}
If $6$ divides the order of $E(K)_{\tors}$, then $E(K)_{\tors}$ is isomorphic to either $\cC_6$, $\cC_{12}$, $\cC_{24}$, $\cC_2 \times \cC_{6}$, $\cC_2 \times \cC_{12}$ or $\cC_3 \times \cC_6$.
\end{lemma}
\begin{proof}
By \cite[Theorem 7 (11)]{GL16}, we see that $E(K)_{\tors}$ cannot contain $\cC_{18}$.

We can see that $E(K)_{\tors}$ does not contain $\cC_6 \times \cC_6$, as $\Q(E[6])$ is a Galois extension of $\Q$, hence it would have to be either $\Q$ or $F$. But it is impossible to have full 6-torsion over $\Q$ or over a quadratic field.

By \cite[Theorem 8]{BN16}, we see that $\cC_3 \times \cC_{12}$ and $\cC_4 \times \cC_{12}$ cannot be contained in $E(K)_{\tors}$.

We claim that $E(K)_{\tors}\supseteq \cC_2 \times \cC_{24}$ is impossible. By a \href{http://matematicas.uam.es/~enrique.gonzalez.jimenez/research/tables/growth/lem8_16a.txt}{\color{blue}search} of the database of $2$-adic images \cite{rouse}, we obtain that the $2$-power torsion of $E$ cannot grow from $\cC_2$ over $\Q$ to $\cC_2 \times \cC_{8}$  in a $D_4$ extension.

It follows that $E(\Q)[2^\infty]\supseteq\cC_2 \times \cC_2$ or $E(\Q)[2^\infty]\supseteq \cC_4$.

Now we split the proof in two cases: either $G_E(3)$ is conjugate to \texttt{3Ns} or not.

We will show that if $E(\Q)[2^\infty]\supseteq\cC_2 \times \cC_2$ or $E(\Q)[2^\infty]\supseteq \cC_4$, then $G_E(3)$ is not conjugate to \texttt{3Ns}. By \cite[Theorem 1.1]{zyw}, $E(\Q)[2^\infty]\supseteq\cC_2 \times \cC_2$ if and only if
$$j(E)=256\frac{(t^2+t+1)^3}{t^2(t+1)^2} \quad \text{ for some }t \in \Q \smallsetminus \{0,-1\}.$$
On the other hand, if $G_E(3)$ is conjugate in $\GL_2(\F_3)$ to \texttt{3Ns}, then by \cite[Theorem 1.2]{zyw}, we have
$$j(E)=s^3 \quad \text{ for some }s \in \Q.$$
Hence we obtain the equation
$$C\,:\,s^3=256\frac{(t^2+t+1)^3}{t^2(t+1)^2},$$
which cuts out a singular genus $1$ curve, which is birationally equivalent to
$$X_1\,:\,y^2=x^3+1,$$
which has rank $0$ and $6$ rational points. We \href{http://matematicas.uam.es/~enrique.gonzalez.jimenez/research/tables/growth/lem8_16b.txt}{\color{blue}obtain} that the only (affine)rational points on $X_1$ satisfy $s=12$, implying $j(E)=1728$. But by \cite[\S 4.4]{clark}, elliptic curves with CM cannot have $\cC_2 \times \cC_{24}$ torsion over quartic number fields.

If $E(\Q)[2^\infty]\supseteq \cC_4$, then $E$ has a 4-isogeny over $\Q$, from which it follows that
\cite[Table 3]{loz}
$$j(E)=\frac{(h^2+16h+16)^3}{h(h+16)} \text{ for some }h \in \Q \smallsetminus \{0,-16\}.$$
Thus, $h(h+16)$ needs to be a cube, giving an equation of an elliptic curve
$$X_2:h^2+16h=v^3.$$
 The elliptic curve $X_2$ \href{http://matematicas.uam.es/~enrique.gonzalez.jimenez/research/tables/growth/lem8_16b.txt}{\color{blue} has} rank 0 over $\Q$ and its torsion points (with the assumption $h\neq 0, -16$) give $j(E)=1728$ and $287496$, both of which are CM values. Again, by \cite[\S 4.4]{clark}, we note that elliptic curves with CM cannot have $\cC_2 \times \cC_{24}$ torsion over quartic number fields, which rules out this possibility.

Thus, we have proved $G_E(3)$ is not conjugate to \texttt{3Ns}. We claim that then $E(F)[3]\neq \{\mathcal O\}$. To see this, in Table \ref{tableSutherland}, we see that the only possibility of $\Q(P)$ being equal to $K$ (which is equivalent to $d_v=4$ for some $v$) is that $G_E(3)$ is conjugate to \texttt{3Cs}. But if $G_E(3)$ is \texttt{3Cs} and $[\Q(P):\Q]=4$ for a point $P$ of order $3$, then it would follow, since $\Q(P)\subseteq \Q(E[3])$ and both $\Q(E[3])$ and $\Q(P)$ are of degree $4$, that $\Q(E[3])=\Q(P)=K$. But $\Q(E[3])$ is Galois over $\Q$ and $K$ is not. We conclude that $E(F)[3]\neq \{\mathcal O\}$.

We now claim that if $E(\Q)[2^\infty]\supseteq\cC_2 \times \cC_2,$ then $E(K)\supseteq \cC_2\times \cC_{24}$ is impossible. Since $E(F)[3]=E(\Q)[3]\oplus E^d(\Q)[3]$, where $F=\Q(\sqrt{d})$, so either $E$ or $E^d$ has a $3$-torsion point; we choose without loss of generality $E$ to be the twist with $3$-torsion. Since quadratic twisting does not change the 2-torsion, we have that $E(\Q)\supseteq \cC_2 \times \cC_6$.  By a \href{http://matematicas.uam.es/~enrique.gonzalez.jimenez/research/tables/growth/lem8_16a.txt}{\color{blue}search} of the database of $2$-adic images \cite{rouse}, we find that there exist $2$-adic representations such that $2$-power torsion of $E$ grows from $\cC_2\times \cC_2$ over $\Q$ to $\cC_2 \times \cC_{8}$ over $F$ or/and over $K$. In the former case we obtain $\cC_2\times\cC_{24}\subseteq E(F)$, which is impossible over  quadratic fields. In the last case, $\cC_2\times \cC_8\subset E(K)[2^\infty]$, we obtain that the field $L=\Q(E[4])$ is a quartic number field. But together with our assumption that $E(\Q)[3]\neq 0$, we have that $E(L)$ contains a subgroup isomorphic to $\cC_4\times \cC_{12}$, which is in contradiction with \cite[Theorem 7]{BN16}.

It remains to prove that if $E(\Q)[2^\infty] \supseteq \cC_4$ then $E(K)$ cannot contain $\cC_2\times \cC_{24}$. By a \href{http://matematicas.uam.es/~enrique.gonzalez.jimenez/research/tables/growth/lem8_16a.txt}{\color{blue}search} of the database of $2$-adic images \cite{rouse}, we find that if $E(\Q)[2^\infty]\supseteq \cC_4$ and $E(K)[2^\infty]\supseteq \cC_2 \times \cC_8$, then it is necessary that already $E(F)[2^\infty]\supseteq \cC_2 \times \cC_8$. But this implies that $E(F)\supseteq \cC_2 \times \cC_{24}$, which is impossible  over quadratic number fields.

Thus we have proved that $\cC_2 \times \cC_{24}$ torsion is impossible over $K$.

As $\cC_2 \times \cC_{24}$ torsion is impossible over $K$, it follows by Lemma \ref{lem2x2n} that $\cC_{48}$ is impossible.

\end{proof}

\begin{proof}[Proof of Theorem \ref{tm_D4}]
Finally, to prove Theorem \ref{tm_D4}, we note that all the groups that are in $\Phi_\Q(2)$ are also in $\Phi_\Q^{\operatorname{D_4}}(4)$ and by \cite{jk-dihedral}, there exists infinitely many elliptic curves with $\cC_{20}$ and $\cC_{24}$ torsion over dihedral quartic number fields.  Therefore $\Phi_\Q(2) \cup\left\{  \cC_{20}\,,\,\cC_{24} \right\}\subseteq \Phi^{\operatorname{D}_4}_\Q(4)$. The equality holds by Lemmas \ref{prop_2}-\ref{6div}.
\end{proof}

\begin{proof}[Proof of Corollary \ref{cor_Q4}] Let $K$ a quartic number field, then the group $\Gal(\widehat{K}/\Q)$ is isomorphic to $\operatorname{S_4}$, $\operatorname{A_4}$, $\operatorname{D_4}$, $\operatorname{V_4}$ or $\cC_4$. Therefore
$$
\Phi_\Q(4)= \Phi^{\operatorname{S}_4}_\Q(4)\,\cup \,\Phi^{\operatorname{A}_4}_\Q(4)\,\cup \,\Phi^{\operatorname{D}_4}_\Q(4)\,\cup \,\Phi^{\operatorname{V}_4}_\Q(4)\,\cup \,\Phi^{\cC_4}_\Q(4).
$$
The cases $\Phi^{\operatorname{V}_4}_\Q(4)$ and $\Phi^{\cC_4}_\Q(4)$ have been determined by Chou \cite{chou}. The remaining cases are dealt with in Corollary \ref{cor_triv} and Theorem  \ref{cor_Q4}.
\end{proof}

\

\noindent {\bf Acknowledgements:} We thank Abbey Bourdon for suggesting Remark \ref{abbey}, Pete Clark for pointing out a mistake in a previous version of the paper and Maarten Derickx for pointing out the example in Remark \ref{remark:maarten}.

\clearpage

\section*{Appendix: Images of Mod $p$ Galois representations associated to elliptic curves over $\Q$}

For each possible known subgroup $G_E(p)\subsetneq \GL_2(\F_p)$ where $E/\Q$ is a non-CM elliptic curve and $p$ is a prime, Tables \ref{tableSutherland} and  \ref{tableSutherland2} list in the first and second column the corresponding labels in Sutherland and Zywina notations, and the following data:
\begin{itemize}
\item $d_v=[G_E(p):G_E(p)_v]=|G_E(p).v|$ for $v\in\F_p^2$, $v\ne (0,0)$; equivalently, the degrees of the extensions $\Q(P)$ over $\Q$ for points $P\in E(\overline \Q)$ of order $p$.
\item$d=|G_E(p)|$; equivalently, the degree $\Q(E[p])$ over $\Q$.
\end{itemize}

Note that Tables \ref{tableSutherland} and  \ref{tableSutherland2} are partially extracted from Table 3 of \cite{Sutherland2}. The difference is that \cite[Table 3]{Sutherland2} only lists the minimum of $d_v$, which is denoted by $d_1$ therein. The \href{http://matematicas.uam.es/~enrique.gonzalez.jimenez/research/tables/growth/Table1and2.txt}{\color{blue} computations} needed to obtain these Tables were performed in Magma.

\begin{table}[h!]
\begin{footnotesize}
\begin{tabular}{cc}

\begin{tabular}{@{\hskip -6pt}cllcc}
& \multicolumn{1}{c}{Sutherland}  &  \multicolumn{1}{c}{Zywina}  &$d_v $ & $d$\\\toprule
&\texttt{2Cs} & $G_1$   & 1 & 1\\
&\texttt{2B} &$G_2$  & 1\,,\,2 & 2\\
&\texttt{2Cn} & $G_3$  & 3  & 3\\
\midrule
&\texttt{3Cs.1.1} &$H_{1,1}$ & 1\,,\,2 & 2\\
&\texttt{3Cs} & $G_1$  & 2\,,\,4  & 4 \\
 &\texttt{3B.1.1} & $H_{3,1}$ & 1\,,\,6  &6 \\
&\texttt{3B.1.2} &  $H_{3,2}$ & 2\,,\,3  & 6\\
&\texttt{3Ns} &$G_2$  & 4  & 8\\
&\texttt{3B} & $G_3$  & 2\,,\,6  & 12\\
&\texttt{3Nn} & $G_4$  & 8  & 16\\
\midrule
&\texttt{5Cs.1.1} & $H_{1,1}$  & 1\,,\,4  & 4\\
&\texttt{5Cs.1.3} & $H_{1,2}$ & 2\,,\,4  & 4\\
&\texttt{5Cs.4.1} & $G_1$  & 2\,,\,4\,,\,8  & 8\\
&\texttt{5Ns.2.1} & $G_3$  & 8\,,\,16  & 16\\
&\texttt{5Cs} &  $G_2$  & 4\,,\,4  & 16\\
&\texttt{5B.1.1} & $H_{6,1}$  & 1\,,\,20  & 20\\
&\texttt{5B.1.2} & $H_{5,1}$  & 4\,,\,5  & 20\\
&\texttt{5B.1.4} & $H_{6,2}$  & 2\,,\,20  & 20\\
&\texttt{5B.1.3} & $H_{5,2}$  & 4\,,\,10 & 20\\
&\texttt{5Ns} & $G_{4}$  & 8\,,\,16  & 32\\
&\texttt{5B.4.1} & $G_{6}$  & 2\,,\,20  & 40\\
&\texttt{5B.4.2} & $G_{5}$  & 4\,,\,10  & 40\\
&\texttt{5Nn} & $G_{7}$  & 24  & 48\\
&\texttt{5B} & $G_{8}$  & 4\,,\,20 & 80\\
&\texttt{5S4} & $G_{9}$ & 24  & 96 \\
\bottomrule
\end{tabular}

&\qquad\quad

\begin{tabular}{@{\hskip -6pt}cllcc}
& \multicolumn{1}{c}{Sutherland}  &  \multicolumn{1}{c}{Zywina}   &$d_v $ & $d$\\\toprule
&\texttt{7Ns.2.1} &  $H_{1,1}$ &   $ 6\,,\,9\,,\,18$  & 18 \\
&\texttt{7Ns.3.1} & $G_{1}$ & $12\,,\,18$ & 36 \\
&\texttt{7B.1.1} & $H_{3,1}$ & $ 1\,,\,42$ & 42 \\
&\texttt{7B.1.3} & $H_{4,1}$ & $ 6\,,\, 7$  & 42 \\
 &\texttt{7B.1.2} & $H_{5,2}$ & $ 3\,,\,42$ & 42 \\
&\texttt{7B.1.5} & $H_{5,1}$ &  $ 6\,,\,21$   & 42 \\
 &\texttt{7B.1.6} & $H_{3,2}$ &  $ 2\,,\,21$ & 42 \\
 &\texttt{7B.1.4} &$H_{4,2}$ &  $ 3\,,\,14$   & 42 \\
 &\texttt{7Ns} & $G_{2}$ &  $ 12\,,\,36$ & 72 \\
  &\texttt{7B.6.1} & $G_{3}$  & $ 2\,,\,42$   & 84 \\
  &\texttt{7B.6.3} & $G_{4}$  & $ 6\,,\, 14$   & 84 \\
  &\texttt{7B.6.2} & $G_{5}$  & $ 6\,,\, 42$   & 84 \\
  &\texttt{7Nn} & $G_{6}$  & $ 48$&    96 \\
   &\texttt{7B.2.1} & $H_{7,2}$&  $ 3\,,\,42$   & 126 \\
   &\texttt{7B.2.3} & $H_{7,1}$ &  $ 6\,,\, 21$   & 126 \\
   &\texttt{7B} & $G_{7}$ & $ 6\,,\, 42$ & 252 \\
\midrule
&\texttt{11B.1.4} & $H_{1,1}$  & 5\,,\,110 & 110\\
&\texttt{11B.1.5} & $H_{2,1}$  & 5\,,\,110  & 110\\
&\texttt{11B.1.6} & $H_{2,2}$  & 10\,,\,55  & 110\\
&\texttt{11B.1.7} & $H_{1,2}$  & 10\,,\,55  & 110\\
&\texttt{11B.10.4} & $G_{1}$  & 10\,,\,110  & 220\\
&\texttt{11B.10.5} & $G_{2}$  & 10\,,\,110  & 220\\
&\texttt{11Nn} & $G_{3}$ & 120 & 240\\
  \bottomrule
 & & & & \\
 & & & & \\
 \end{tabular}

\bigskip
\smallskip
\end{tabular}
\end{footnotesize}
\caption{Possible images $G_E(p)\ne \GL_2(\F_p)$, for $p\le 11$, for non-CM elliptic curves $E/\Q$.}\label{tableSutherland}
\vspace{10pt}
\end{table}
\smallskip
\begin{table}
\begin{footnotesize}
\begin{tabular}{cc}

\begin{tabular}{@{\hskip -6pt}cllcc}
& \multicolumn{1}{c}{Sutherland}  &  \multicolumn{1}{c}{Zywina}  &$d_v $ & $d$\\\toprule
&\texttt{13S4} &$G_{7}$ & $72\,,\,96$ & 288 \\
 &\texttt{13B.3.1} &$H_{5,1}$ & $3 \,,\, 156$ & 468 \\
 &\texttt{13B.3.2} & $H_{4,1}$ & $ 12\,,\, 39$ & 468 \\
 &\texttt{13B.3.4} & $H_{5,2}$ & $ 6\,,\,156$ & 468 \\
 &\texttt{13B.3.7} & $H_{4,2}$ & $ 12 \,,\, 78$ & 468 \\
  &\texttt{13B.5.1} & $G_{2}$ & $ 4 \,,\, 156$ & 624 \\
  &\texttt{13B.5.2} & $G_{1}$ & $ 12 \,,\, 52$ & 624\\
  &\texttt{13B.5.4} & $G_{3}$ & $ 12 \,,\, 156$ & 624 \\
  &\texttt{13B.4.1} & $G_{5}$ & $ 6 \,,\, 156$  & 936 \\
  &\texttt{13B.4.2} & $G_{4}$ & $ 12 \,,\, 78$   & 936 \\
  &\texttt{13B} & $G_{6}$ & $ 12 \,,\, 156$  & 1872 \\
  \bottomrule
\end{tabular}

&\qquad\quad

\begin{tabular}{@{\hskip -6pt}cllcc}
& \multicolumn{1}{c}{Sutherland}  &  \multicolumn{1}{c}{Zywina}   &$d_v $ & $d$\\\toprule
&\texttt{17B.4.2} & $G_{1}$ & $ 8 \,,\, 272$  & 1088 \\
&\texttt{17B.4.6} & $G_{2}$ & $ 16 \,,\, 136$  & 1088 \\
\midrule
 &\texttt{37B.8.1} & $G_{1}$ & 12 \,,\,  1332 & 15984 \\
 &\texttt{37B.8.2} & $G_{2}$ &  36 \,,\, 444 & 15984 \\
   \bottomrule
 & & & & \\
 & & & & \\
 & & & & \\
 & & & & \\
 & & & & \\
 & & & & \\
 \end{tabular}

\bigskip
\smallskip
\end{tabular}

\end{footnotesize}
\caption{Known images $G_E(p)\ne \GL_2(\F_p)$, for $p=13,17$ or $37$, for non-CM elliptic curves $E/\Q$.}\label{tableSutherland2}
\vspace{10pt}
\end{table}

\end{document}